\documentclass[amstex,12pt,reqno]{amsart}
\usepackage{amsmath, amssymb, amsthm}
\usepackage{mathrsfs}
\usepackage[all]{xy}

\newtheorem{theorem}{Theorem}[section]
\newtheorem{lemma}[theorem]{Lemma}
\newtheorem{proposition}[theorem]{Proposition}
\newtheorem{corollary}[theorem]{Corollary}

\theoremstyle{definition}

\newtheorem*{remark}{Remark}

\title
[The $p$-Weil--Petersson Teichm\"uller space]
{The $p$-Weil--Petersson Teichm\"uller space \\and the quasiconformal extension of curves}

\author[H. Wei]{Huaying Wei} 
\address{Department of Mathematics and Statistics, Jiangsu Normal University \endgraf Xuzhou 221116, PR China} 
\curraddr{Department of Mathematics, School of Education, Waseda University \endgraf
Shinjuku, Tokyo 169-8050, Japan}
\email{hywei@jsnu.edu.cn} 

\author[K. Matsuzaki]{Katsuhiko Matsuzaki}
\address{Department of Mathematics, School of Education, Waseda University \endgraf
Shinjuku, Tokyo 169-8050, Japan}
\email{matsuzak@waseda.jp}

\makeatletter
\@namedef{subjclassname@2020}{%
\textup{2020} Mathematics Subject Classification}
\makeatother

\thanks{Research supported by 
Japan Society for the Promotion of Science (KAKENHI 18H01125 and 21F20027).}

\begin{document}

\maketitle

\begin{abstract}
We consider the correspondence between the space of $p$-Weil--Petersson curves $\gamma$ on the plane
and the $p$-Besov space of $u=\log \gamma'$ on the real line for $p>1$. We prove that the variant of the
Beurling--Ahlfors extension defined by using the heat kernel yields a holomorphic map for $u$ on a domain of
the $p$-Besov space to the space of $p$-integrable Beltrami coefficients. 
This in particular gives a global real-analytic section for the Teichm\"uller projection from the 
space of $p$-integrable Beltrami coefficients to the $p$-Weil--Petersson Teichm\"uller space.
\end{abstract}

\section{Introduction}
\subsection{Background on the Weil--Petersson class and its generalization}
An increasing homeomorphism $h$ of the real line $\mathbb R$ onto itself belongs to the {\it $2$-Weil--Petersson class} on $\mathbb R$ (nowadays this is usually called the Weil--Petersson class in the literature, but we add the index $2$ here 
for its generalization)  
if, by definition, it can be extended to a quasiconformal homeomorphism of the upper half-plane $\mathbb U$ 
onto itself whose Beltrami coefficient is $2$-integrable 
in the hyperbolic metric (the index $2$ actually comes from here). 
Let $W_2(\mathbb R)$ be the set of all normalized $2$-Weil--Petersson class homeomorphisms on $\mathbb R$ which keeps $0$, $1$ and $\infty$ fixed. This is the real model of the {\it $2$-Weil--Petersson Teichm\"uller space}. 

A study of the $2$-Weil--Petersson Teichm\"uller space was initiated by Cui \cite{Cu} where he gave some characterizations of the $2$-Weil--Petersson class and showed that this is the completion of the set of all normalized $C^{\infty}$-diffeomorphisms under the Weil--Petersson metric. Later, Takhtajan and Teo \cite{TT} studied systematically the $2$-Weil--Petersson Teich\-m\"ul\-ler space. They proved that it is the connected component of the identity in the universal Teichm\"uller space viewed as a complex Hilbert manifold and established many other equivalent characterizations of the $2$-Weil--Petersson Teichm\"uller space. Moreover, they proposed a problem for characterizing intrinsically the elements in the $2$-Weil--Petersson class without using quasiconformal extension. Then, Shen and his coauthors \cite{Sh18, ST, SW} did among other work solve this problem by characterizing the $2$-Weil--Petersson class directly in terms of the fractional dimensional Sobolev space $H_{\mathbb R}^{1/2}$ of real-valued functions. Recently, Bishop \cite{Bi} gave lots of new characterizations of the $2$-Weil--Petersson class which link to various concepts in geometric measure theory and hyperbolic geometry. In addition, this work has motivations from string theory and SLE theory, and in reverse,
it has applications to these theories (see \cite{Bi, Wa} and references therein).

For $p > 1$, the normalized {\it $p$-Weil--Petersson class} $W_p(\mathbb R)$ can be defined similarly by just changing $2$-integrability into $p$-integrability. The generalization apart from the case of $p = 2$ is natural, and several works have been done in this direction (see \cite{Gu, Mat, Mat1, TS, Ya}). These generalizations are usually straightforward, but there are really a few crucial differences in the arguments between the cases of $p=2$ and $p \neq 2$. 

In the present paper, we study the Weil--Petersson theory of the universal Teich\-m\"ul\-ler space, and mainly consider a $p$-Weil--Petersson curve from $\mathbb R$ into the whole plane $\mathbb C$, which is the generalization of
a $p$-Weil--Petersson class homeomorphism. 
Here, by saying a curve, 
we include its parametrization, which has more information than just the image of a curve. 
We will prove the existence of a canonical quasiconformal extension of a Weil--Petersson curve to $\mathbb C$ using
the variant of the Beurling--Ahlfors extension by the heat kernel introduced in Fefferman, Kenig and Pipher \cite{FKP}. 
Its detailed exposition is in our previous paper \cite{WM-2}.
Then, by the restriction of this quasiconformal extension operator to the $p$-Weil--Petersson class $W_p(\mathbb R)$, 
we can obtain novel results and also the reformation of the existing results on 
the $p$-Weil--Petersson Teichm\"uller space. It is worthwhile to mention that our new approach is natural for the investigation of absolutely continuous curves on $\mathbb R$ induced by
quasiconformal mappings of $\mathbb C$, and can be used for other problems.

\subsection{Parametrization of Weil--Petersson curves}
Taking the space ${\mathcal M}_p(\mathbb U)$ of the Beltrami coefficients
that are $p$-integrable with respect to the hyperbolic metric on $\mathbb U$,
the $p$-Weil--Petersson Teichm\"uller space $T_p(\mathbb U)$ is given by the 
Teichm\"uller projection $\pi:{\mathcal M}_p(\mathbb U) \to T_p(\mathbb U)$. It is known that $T_p(\mathbb U)$ has a unique complex Banach manifold structure via the Bers embedding through the Schwarzian derivative (or via the logarithmic derivative embedding) such that the Teichm\"uller projection $\pi$ is holomorphic with local holomorphic inverse for $p \geq 2$
(see \cite{TS, WM-4}). This can be considered on 
the lower half-plane $\mathbb L$ in the same way. 

{\it A $p$-Weil--Petersson curve} $\gamma:\mathbb R \to \mathbb C$ is the restriction of
a quasiconformal homeomorphism of $\mathbb C$ whose complex dilatation on $\mathbb U$ and $\mathbb L$
belongs to ${\mathcal M}_p(\mathbb U)$ and ${\mathcal M}_p(\mathbb L)$, respectively. 
We impose the normalization $\gamma (0) = 0$ and $\gamma (1) = 1$ and $\gamma (\infty) = \infty$ on 
every $p$-Weil--Petersson curve $\gamma$. Let ${\rm WPC}_p$ be the set of all normalized $p$-Weil--Petersson curves. 
Hence, the space ${\rm WPC}_p$ can be understood 
in the spirit of the Bers simultaneous uniformization so that 
${\rm WPC}_p$ is identified with the product of the $p$-Weil--Petersson Teichm\"uller spaces
$T_p(\mathbb U) \times T_p(\mathbb L)$,  which endows ${\rm WPC}_p$ with the product complex Banach manifold structure. 
In our recent paper, we have proved the following. 

\begin{theorem}[\cite{WM-4}]\label{thm1}
For any normalized $p$-Weil--Petersson curve $\gamma \in {\rm WPC}_p$, the logarithm of the derivative $\log \gamma'$ belongs to
the $p$-Besov space $B_p(\mathbb R)$. Moreover, this correspondence $L:{\rm WPC}_p \to B_p(\mathbb R)$
is a biholomorphic homeomorphism onto its image.
\end{theorem}

In the present paper, we will derive the inverse of $L$ in Theorem \ref{thm1} by
constructing the quasiconformal mappings explicitly. 
Precisely,  we show that if $\log \gamma'$ is in some
neighborhood $U(B_p^{\mathbb R}(\mathbb R))$ of the real Banach subspace $B_p^{\mathbb R}(\mathbb R)$ 
consisting of all real-valued $p$-Besov functions, then
the variant of the Beurling--Ahlfors extension by the heat kernel of $\gamma$ to both $\mathbb U$ and $\mathbb L$ has complex dilatations
in $\mathcal M_p(\mathbb U)$ and $\mathcal M_p(\mathbb L)$. Moreover, this correspondence $\widetilde \Lambda$ is holomorphic.
Then, if we further take the composition with the product of the Teichm\"uller projections
$\widetilde \pi:\mathcal M_p(\mathbb U) \times \mathcal M_p(\mathbb L) \to T_p(\mathbb U) \times T_p(\mathbb L)$,
this gives the inverse of $L:{\rm WPC}_p \cong T_p(\mathbb U) \times T_p(\mathbb L) \to B_p(\mathbb R)$ on the neighborhood $U(B_p^{\mathbb R}(\mathbb R))$.

\begin{theorem}[see Theorem \ref{qcU}] \label{thm5}
There is a holomorphic map 
$$
\widetilde \Lambda:U(B_p^{\mathbb R}(\mathbb R)) \to {\mathcal M}_p(\mathbb U) \times {\mathcal M}_p(\mathbb L)
$$ 
defined on a neighborhood $U(B_p^{\mathbb R}(\mathbb R)) \subset B_p(\mathbb R)$ of 
$B_p^{\mathbb R}(\mathbb R)$
such that $L \circ \widetilde \pi \circ \widetilde \Lambda$ is the identity on $U(B_p^{\mathbb R}(\mathbb R))$.
\end{theorem}

The point of this consequence is that a single formula of the the variant of the Beurling--Ahlfors extension by the heat kernel
can be applied to all complex-valued functions in some neighborhood of the real-valued $p$-Besov functions.
Other quasiconformal extensions used in the literature are not known to have this property.

\subsection{Applications to the Weil--Petersson class}
We apply Theorem \ref{thm5} restricted to $B_p^{\mathbb R}(\mathbb R)$ itself in
$U(B_p^{\mathbb R}(\mathbb R))$
(or restricted to $W_p(\mathbb R)$ in ${\rm WPC}_p$).  
This produces new assertions on the $p$-Weil--Petersson Teichm\"uller space.

The first implication of Theorem \ref{thm5} is the following, which is
just a special case of this theorem. 
\begin{corollary}\label{specialcase}
The holomorphic map $\widetilde \Lambda$ 
sends $u \in B_p^{\mathbb R}(\mathbb R)$ to a symmetric pair of Beltrami coefficients
$\left( \mu_u(z),\, \overline{\mu_u(\bar z)}\right) \in {\mathcal M}_p(\mathbb U) \times {\mathcal M}_p(\mathbb L)$,
and hence the normalized curve 
$\gamma_u(x) = \left(\int_0^1 e^{u(t)}dt \right)^{-1}\int_0^x e^{u(t)} dt$
belongs to the $p$-Weil--Petersson class $W_p(\mathbb R)$. 
Moreover, the correspondence $B_p^{\mathbb R}(\mathbb R) \to W_p(\mathbb R)$ given by $u \mapsto \gamma_u$
and its inverse are real-analytic homeomorphisms.
\end{corollary}

The statements that $\gamma_u$ belongs to $W_2(\mathbb R)$ if $u \in B_2^{\mathbb R}(\mathbb R)$ 
and that $B_2^{\mathbb R}(\mathbb R) \to W_2(\mathbb R)$ (and its inverse) is real-analytic were proved in Shen and Tang
\cite{ST} by using a modified Beurling--Ahlfors extension due to Semmes \cite{Se}
which can work only for such $u$ with small norm. In this case,
the quasiconformal extensions after decomposing $\gamma_u$ 
into small norm pieces and then the composition of such extensions are required. On the contrary,
our extension has a better property (no norm assumption on $u$ is needed) and can be applied for any $p>1$. 
Moreover, as the advantage of the one-time extension by a single formula,
it holds several desirable properties of its complex dilatation, and also it yields the next.

The second implication of Theorem \ref{thm5} is that the variant of the Beurling--Ahlfors extension by the heat kernel
yields a global real-analytic section $\Lambda \circ L$ for the $p$-Weil--Petersson Teichm\"uller space
$T_p \cong W_p(\mathbb R) \subset {\rm WPC}_p$, where $\Lambda$ is the diagonal reduction of $\widetilde \Lambda$.

\begin{corollary}\label{cor6}
Under the identification of $W_p(\mathbb R)$ with $T_p$,
the map 
$$
\Lambda \circ L|_{W_p(\mathbb R)}:T_p \to {\mathcal M}_p
$$
is a global real-analytic section for the Teichm\"uller projection $\pi:{\mathcal M}_p \to T_p$. 
\end{corollary}
For the universal Teichm\"uller space $T$ and its subspaces invariant under Fuchsian groups,
the Douady--Earle extension in \cite{DE} gives a global real-analytic section for the Teichm\"uller projection. Here, the result of Corollary \ref{cor6} is a counterpart for $T_p$.
From Corollary \ref{cor6}, we also see that $T_p$ is contractible 
since ${\mathcal M}_p$ is contractible. 
A contraction $\phi:T_p\times [0, 1] \to T_p$ is given explicitly by
$\phi(h, t)= \pi\left((1 - t)\Lambda \circ L|_{W_p(\mathbb R)} (h)\right)$. 
The holomorphic contractibility of $T_2$, which means that the contraction $\phi(\cdot,t)$ is holomorphic 
for each fixed $t \in [0, 1]$,
was obtained by Fan and Hu \cite{FH} though this does not imply the existence of a global holomorphic section for $\pi$.

\subsection{Plan of this paper}
We end this introduction section with the organization of the paper. In Section 2, we recall definitions and properties of the BMO space, Muckenhoupt weights, and the Besov space, and prepare several basic results for later use. 
In Section 3, we give a detailed exposition on the variant of the Beurling-Ahlfors extension by the heat kernel 
for complex-valued BMO functions. This extension plays an important role in the proof of our main theorem (Theorem \ref{qcU}). 
Section 4 is devoted to this proof and its consequences as we described above. In Section 5 as an appendix, 
we show that our extension translated to the setting of the unit circle
also yields the desired quasiconformal extension to the unit disk.

\section{Preliminaries on BMO, ${\rm A}_{\infty}$-weights, and the Besov space}

A locally integrable complex-valued function $u$ on $\mathbb R$ is of {\it BMO} if
$$
\Vert u \Vert_*=\sup_{I \subset \mathbb R}\frac{1}{|I|} \int_I |u(x)-u_I| dx <\infty,
$$
where the supremum is taken over all bounded intervals $I$ on $\mathbb R$ and $u_I$ denotes the integral mean of $u$
over $I$. The set of all BMO functions on $\mathbb R$ is denoted by ${\rm BMO}(\mathbb R)$.
This is regarded as a Banach space with norm $\Vert \cdot \Vert_*$
by ignoring the difference of constant functions.
It is said that $u \in {\rm BMO}(\mathbb R)$ is of {\it VMO} if
$$ 
\lim_{|I| \to 0}\frac{1}{|I|} \int_I |u(x)-u_I| dx=0,
$$
and the set of all such functions is denoted by ${\rm VMO}(\mathbb R)$.
This is a closed subspace of ${\rm BMO}(\mathbb R)$.
The {\it John--Nirenberg inequality} for BMO functions (see \cite[VI.2]{Ga}, \cite[IV.1.3]{St2}) asserts that
there exists two universal positive constants $C_0$ and $C_{JN}$ such that for any  BMO function $u$, 
any bounded interval $I$ of $\mathbb{R}$, and any $\lambda > 0$, it holds that
\begin{equation}\label{JN}
\frac{1}{|I|} |\{t \in I: |u(t) - u_I| \geq \lambda \}| \leq C_0 \exp\left(\frac{-C_{JN}\lambda}{\Vert u \Vert_*} \right).
\end{equation}

A locally integrable non-negative measurable function $\omega \geq 0$ on $\mathbb R$ 
is called a {\it weight}. We say that $\omega$ is
an {\it $A_p$-weight} of Muckenhoupt \cite{M} for $p>1$ if there exists a constant $C_p(\omega) \geq 1$ such that
\begin{equation}\label{Ap}
\left(\frac{1}{|I|} \int_I \omega(x)dx \right)\left(\frac{1}{|I|} \int_{I} \left(\frac{1}{\omega(x)}\right)^{\frac{1}{p-1}}dx\right)^{p-1}
\leq C_p(\omega)
\end{equation}
for any bounded interval $I \subset \mathbb R$. We call the optimal value of such $C_p(\omega)$
the $A_p$-constant for $\omega$. We define $\omega$ to be an {\it $A_\infty$-weight} if $\omega$ is an
$A_p$-weight for some $p>1$, that is, $A_\infty=\bigcup_{p>1} A_p$.
It is known that $\omega$ is an $A_\infty$-weight if and only if
there are positive constants $\alpha(\omega)$, $K(\omega)>0$ such that  
\begin{equation}\label{SD}
\frac{\int_E \omega(x)dx}{\int_I \omega(x)dx}\leq K(\omega)\left(\frac{|E|}{|I|}\right)^{\alpha(\omega)}
\end{equation}
for any bounded interval $I \subset \mathbb{R}$ and 
for any measurable subset $E \subset I$ (see \cite[Theorem V]{CF}
and \cite[Lemma VI.6.11]{Ga}).

The Jensen inequality implies that
\begin{equation}\label{Jensen}
\exp \left(\frac{1}{|I|} \int_I \log \omega(x) dx \right) 
\leq \frac{1}{|I|} \int_I \omega(x) dx.
\end{equation}
Another characterization of $A_\infty$-weights can be given by 
the inverse Jensen inequality. 
Namely,
$\omega \geq 0$
belongs to the class of $A_\infty$-weights 
if and only if there exists a constant $C_\infty(\omega) \geq 1$ such that
\begin{equation}\label{iff}
\frac{1}{|I|} \int_I \omega(x) dx \leq C_\infty(\omega) \exp \left(\frac{1}{|I|} \int_I \log \omega(x) dx \right) 
\end{equation}
for every bounded interval $I \subset \mathbb R$ (see \cite{Hr}). 
We call the optimal value of such $C_\infty(\omega)$
the $A_\infty$-constant for $\omega$. 
If $\omega$ is an $A_p$-weight, then $C_\infty(\omega) \leq C_p(\omega)$ by the Jensen inequality.
If $\omega$ is an $A_\infty$-weight,
the constants $\alpha(\omega)$ and $K(\omega)$ in (\ref{SD}) are estimated by $C_\infty(\omega)$ 
as is shown in \cite[Theorem 1]{Hr},
and 
$C_p(\omega)$ and $p$ are estimated by $\alpha(\omega)$ and $K(\omega)$
(see \cite[Section 3]{CF}).
One can also refer to \cite[p.218]{St2} for these implications.

For a convenience of reference later, we verify  
inequality (\ref{iff}) showing the dependence of $C_\infty(\omega)$ on $\omega$
when $\Vert \log \omega \Vert_*$
is sufficiently small. In particular,
$\omega$ is an $A_\infty$-weight in this case.
Conversely, for any $A_\infty$-weight $\omega$, we have
$\log \omega \in \rm BMO(\mathbb R)$ (see \cite[Lemma VI.6.5]{Ga}).

\begin{proposition}\label{C_0}
Suppose that a weight $\omega \geq 0$ satisfies $\log \omega \in {\rm BMO}(\mathbb R)$.
If the BMO norm $\Vert \log \omega \Vert_*$ 
is less than the constant $C_{JN}$, then $\omega$ is in $A_2 \subset A_\infty$ and the $A_2$- and $A_\infty$-constants
depend only on $\Vert \log \omega \Vert_*$ and tend to $1$ as $\Vert \log \omega \Vert_* \to 0$.
\end{proposition}

\begin{proof}
Let $u=\log \omega \in {\rm BMO}(\mathbb R)$.
For any bounded interval $I \subset \mathbb R$, the John--Nirenberg inequality (\ref{JN}) yields that
\begin{equation}\label{basicJN}
\begin{split}
\frac{1}{|I|} \int_I e^{|u(x)-u_I|}dx
&=\int_0^\infty \frac{1}{|I|} |\{ x \in I: |u(x)-u_I|>\lambda\}| e^\lambda d\lambda+1 \\
& \leq C_0 \int_0^\infty \exp\left(\frac{-C_{JN}\lambda}{\Vert u \Vert_*}\right) e^\lambda d\lambda+1\\
&= \frac{C_0 \Vert u \Vert_*}{C_{JN}-\Vert u \Vert_*}+1
\end{split}
\end{equation}
when $\Vert u \Vert_*<C_{JN}$. 
We set the right side of the above inequality as $C(\omega)^{1/2} \geq 1$, which tends to $1$ 
as $\Vert u \Vert_* \to 0$. From this inequality,
we have
$$
\frac{1}{|I|} \int_I e^{u(x)-u_I}dx \leq C(\omega)^{1/2};\quad
\frac{1}{|I|} \int_I e^{u_I-u(x)}dx \leq C(\omega)^{1/2}.
$$
Hence, we see that $\omega$ is an $A_2$-weight by
\begin{equation*}
\begin{split}
\left(\frac{1}{|I|} \int_I \omega(x) dx \right)\left(\frac{1}{|I|} \int_I \frac{1}{\omega(x)}dx \right)&=
\left(\frac{1}{|I|} \int_I e^{u(x)}dx \right)\left(\frac{1}{|I|} \int_I e^{-u(x)}dx \right)\\
&=\left(\frac{1}{|I|} \int_I e^{u(x)-u_I}dx \right)\left(\frac{1}{|I|} \int_I e^{u_I-u(x)}dx \right) \leq C(\omega).
\end{split}
\end{equation*}
Moreover, the Jensen inequality (\ref{Jensen}) implies that
$$
\frac{1}{|I|} \int_I \frac{1}{\omega(x)}dx \geq \exp \left(\frac{-1}{|I|} \int_I \log \omega(x) dx \right), 
$$
which shows that inequality (\ref{iff}) is satisfied for the constant $C(\omega)$. 
\end{proof}

The {\it $p$-Besov space} $B_p(\mathbb R)$ for $p>1$ is
the set of all measurable complex-valued functions $u$ on $\mathbb R$ that satisfy
$$
\Vert u \Vert^p_{B_p}=\int_{-\infty}^{\infty}\!\int_{-\infty}^{\infty} \frac{|u(t)-u(s)|^p}{|t-s|^2} dsdt<\infty.
$$
For the case of $p = 2$, $B_2(\mathbb R)$ coincides with the Sobolev space $H^{1/2}(\mathbb R)$. 
It is easy to see that if $u \in B_p(\mathbb R)$ then $|u|, {\rm Re}\,u, {\rm Im}\,u \in B_p^{\mathbb R}(\mathbb R)$. Here, $B_p^{\mathbb R}(\mathbb R)$ denotes the set of all real-valued $p$-Besov functions. 
As in the case of BMO functions, $B_p(\mathbb R)$ can be regarded as a Banach space with norm $\Vert \cdot \Vert_{B_p}$
by modulo of constant functions. In other words, we regard $B_p(\mathbb R)$ as
a homogeneous Besov space, which is often denoted by $\dot B_p(\mathbb R)$ in the literature.

The following relation between $B_{p}(\mathbb R)$ and ${\rm VMO}(\mathbb R)$ is important
throughout this paper. We can also find this for $p=2$ in \cite[Section 3]{SW}.

\begin{proposition}\label{BMOnorm}
The inclusion $B_{p}(\mathbb R) \subset {\rm VMO}(\mathbb R)$ holds.
Moreover, $\Vert u \Vert_* \leq \Vert u \Vert_{B_p}$
for every $u \in B_{p}(\mathbb R)$, and in particular, the inclusion map is continuous.
\end{proposition}

\begin{proof}
Let $I \subset \mathbb R$ be any bounded interval. Then,
\begin{align*}
\frac{1}{|I|}\int_I |u(t)-u_I|dt&=\frac{1}{|I|}\int_I \left|u(t)-\frac{1}{|I|}\int_I u(s)ds \right|dt\\
&\leq \frac{1}{|I|^2}\int_I \!\int_I|u(t)-u(s)|ds dt \\
&\leq \left(\frac{1}{|I|^2}\int_I \!\int_I|u(t)-u(s)|^pds dt \right)^{1/p}
\leq \left(\int_I \!\int_I \frac{|u(t)-u(s)|^p}{|t-s|^2}ds dt \right)^{1/p}.
\end{align*}
This implies that $\Vert u \Vert_* \leq \Vert u \Vert_{B_{p}}$. If $u \in B_{p}(\mathbb R)$,
then $|u(t)-u(s)|^p/|t-s|^2$ is integrable on $\mathbb R^2$. Hence, for any $\varepsilon>0$,
there is some $\delta>0$ such that if $|I|<\delta$ then its integral over $I \times I$ is less than 
$\varepsilon$. This shows that $u \in {\rm VMO}(\mathbb R)$.
\end{proof}

Moreover, we see that each element of $B_p^{\mathbb R}(\mathbb R)$ corresponds to an $A_\infty$-weight.

\begin{proposition}\label{BpA}
If a weight $\omega \geq 0$ satisfies that $\log \omega \in B_p^{\mathbb R}(\mathbb R)$, then
$\log \omega$ is in the closure of $L^\infty(\mathbb R)$ in the BMO norm. In particular,
$\omega \in A_2 \subset A_\infty$.
\end{proposition}

\begin{proof}
Let $u=\log \omega$. For any $N>0$, we set $u_N(x)=\max\{\min\{u(x),N\},-N\}$, which belongs to
$L^\infty(\mathbb R)$. We can easily check that
$$
|(u-u_N)(s)-(u-u_N)(t)| \leq |u(s)-u(t)|
$$
for almost all $(s, t) \in \mathbb R^2$. Then, the dominated convergence theorem implies that
$$
\lim_{N \to \infty} \Vert u-u_N \Vert^p_{B_p}=
\lim_{N \to \infty} \int_{-\infty}^{\infty}\!\int_{-\infty}^{\infty} \frac{|(u-u_N)(t)-(u-u_N)(s)|^p}{|t-s|^2}dsdt=0.
$$
Since $\Vert u-u_N \Vert_* \leq \Vert u-u_N \Vert_{B_p}$ by Proposition \ref{BMOnorm}, we see that $u_N \in L^\infty(\mathbb R)$
converges to $u$ in the BMO norm.
In this case, we have $\omega \in A_2$.
See 
\cite[Chap.IV, Theorem 5.14]{GR}.
\end{proof}

We prepare the following three claims on $\rm BMO(\mathbb R)$ and $B_p(\mathbb R)$,
which play key roles in the next two sections. 
Let $I(x,y) \subset \mathbb R$ be the interval $(x-y,x+y)$ for any $x \in \mathbb R$ and $y>0$.

\begin{proposition}\label{prep1}
Let $u$ and $\varphi$ be 
complex-valued functions on $\mathbb R$ 
such that $u \in \rm BMO(\mathbb R)$ and $|\varphi(x)| \leq Ce^{-|x|}$ for some constant $C>0$.
Let $k>0$.
Then,
$$
\int_{\mathbb{R}}|\varphi_y(x-t)||u(t)-u_{I(x,y)}|^kdt \leq C(k) \Vert u \Vert_*^k
$$
for some constant $C(k)>0$.
\end{proposition}

\begin{proof}
For any integer $n \geq 0$, we consider the average $u_{I_n}$
of $u$ over the interval $I_n=(x-2^ny, x+2^ny)$, where $I_0=I(x,y)$.
Using an inequality  
$$
|u_{I_n}-u_{I_{n-1}}| \leq \frac{1}{|I_{n-1}|}\int_{I_{n-1}}|u-u_{I_n}|dt
\leq \frac{2}{|I_n|}\int_{I_n}|u-u_{I_n}|dt \leq 2 \Vert  u \Vert_* 
$$
repeatedly, we have
\begin{equation}\label{n+0}
|u_{I_n}-u_{I_{0}}| \leq 2n \Vert  u \Vert_*.
\end{equation}

Dividing the integral over $\mathbb R$ into those on dyadic intervals, we obtain
\begin{equation}\label{basic}
\begin{split}
&\quad \int_{\mathbb{R}}|\varphi_y(x-t)||u(t)-u_{I(x,y)}|^kdt\\
& = \int_{|x-t|<y}|\varphi_y(x-t)||u(t)-u_{I_0}|^kdt + \sum_{n=0}^{\infty}\int_{2^ny\leq|x-t| 
< 2^{n+1}y}|\varphi_y(x-t)||u(t)-u_{I_0}|^kdt\\
&\leq \frac{C}{y}\int_{|x-t|<y}e^{-\frac{|x-t|}{y}}|u(t)-u_{I_0}|^k dt 
+ \sum_{n=0}^{\infty}\frac{C}{y}\int_{2^ny\leq|x-t| < 2^{n+1}y}e^{-\frac{|x-t|}{y}}|u(t)-u_{I_0}|^kdt\\
&\leq \frac{2C}{|I_0|}\int_{I_0}|u(t)-u_{I_0}|^kdt 
+ \sum_{n=0}^{\infty}\frac{2^{n+2}C}{e^{2^{n}}|I_{n+1}|}\int_{I_{n+1}}|u(t)- u_{I_0}|^kdt\\
&\leq  
C\sum_{n=0}^{\infty}\frac{2^{n+2}}{e^{2^{n-1}}|I_{n}|}\int_{I_{n}}(|u(t)-u_{I_{n}}| + |u_{I_{n}} - u_{I_0}| )^kdt\\
&\leq C\sum_{n=0}^{\infty}\frac{2^{n+2+k}}{e^{2^{n-1}}}\left(\frac{1}{|I_{n}|}\int_{I_{n}}|u(t)-u_{I_{n}}|^kdt
+\frac{1}{|I_{n}|}\int_{I_{n}}|u_{I_{n}} - u_{I_0}|^kdt\right).
\end{split}
\end{equation}
For the last inequality above, we have used $(a + b)^k \leq 2^{k}(a^k + b^k)$ for $a,b \geq 0$.

Here, by the John--Nirenberg inequality (\ref{JN}), we have 
\begin{equation}\label{k!}
\begin{split}
&\quad \frac{1}{|I_{n}|}\int_{I_{n}}|u(t)-u_{I_{n}}|^kdt\\
&=k\int_0^\infty \frac{1}{|I_n|}|\{t \in I_n \mid |u(t)-u_{I_n}|>\lambda\}|\,\lambda^{k-1} d\lambda\\
&\leq k \int_0^\infty C_0 \exp\left(\frac{-C_{JN} \lambda}{\Vert u \Vert_*}\right) \lambda^{k-1} d\lambda
=\frac{C_0 \Gamma(k+1)}{C_{JN}^k} \Vert u \Vert_*^k.
\end{split}
\end{equation}
Moreover, (\ref{n+0}) yields 
\begin{equation}\label{2n^k}
\begin{split}
\frac{1}{|I_{n}|}\int_{I_{n}}|u_{I_{n}} - u_{I_0}|^kdt \leq (2n)^k \Vert u \Vert_*^k.
\end{split}
\end{equation}
Plugging these inequalities into the last line of (\ref{basic}), we obtain the required estimate.
\end{proof}

\begin{proposition}\label{prep2}
Let $u$ and $\varphi$ be 
complex-valued functions on $\mathbb R$ 
such that $u \in \rm BMO(\mathbb R)$ and $|\varphi(x)| \leq Ce^{-|x|}$ for some constant $C>0$.
If $\Vert u \Vert_* < C_{JN}$, 
then
$$
\int_{\mathbb{R}}|\varphi_y(x-t)|e^{|u(t)-u_{I(x,y)}|}dt \leq C(u)
$$
for a constant $C(u)$ given in terms of $\Vert u \Vert_*$. 
\end{proposition}

\begin{proof}
As in the proof of Proposition \ref{prep1}, we have
\begin{equation}\label{2cases}
\begin{split}
&\quad \int_{\mathbb R} |\varphi_y (x-t)|e^{|u(t)-u_{I(x,y)}|} dt\\
&\leq \frac{C}{y}\int_{|x-t|<y}e^{-\frac{|x-t|}{y}}e^{|u(t)-u_{I_0}|}dt 
+ \sum_{n=0}^{\infty}\frac{C}{y}\int_{2^ny\leq|x-t| < 2^{n+1}y}e^{-\frac{|x-t|}{y}}e^{|u(t)-u_{I_0}|}dt\\
 &\leq  
C\sum_{n=0}^{\infty}\frac{2^{n+2}e^{2n \Vert u \Vert_*}}{e^{2^{n-1}}|I_{n}|}\int_{I_{n}}e^{|u(t)-u_{I_{n}}|}dt.
\end{split}
\end{equation}
If $\Vert u \Vert_*< C_{JN}$, then the John--Nirenberg inequality as in (\ref{basicJN}) yields
\begin{equation*}
\begin{split}
\frac{1}{|I_n|}\int_{I_{n}}e^{|u(t)-u_{I_{n}}|}dt
&\leq\frac{C_0 \Vert u \Vert_*}{C_{JN}- \Vert u \Vert_*}+1. 
\end{split}
\end{equation*}
Thus, we obtain the statement of the proposition.
\end{proof}

\begin{lemma}\label{nearreal}
For each $u_0 \in B_p^{\mathbb R}(\mathbb R)$,
there exists a constant $C(u_0)>0$ such that every 
 $u \in B_p(\mathbb R)$ with $\Vert u-u_0 \Vert_{B_p} \leq C_{JN}/4$
satisfies 
$$
\frac{1}{|I|}\int_{I}e^{|u(t)-u_{I}|}dt \leq C(u_0)
$$
for all bounded intervals $I \subset \mathbb R$. Likewise,
$$
\int_{\mathbb{R}}|\varphi_y(x-t)|e^{|u(t)-u_{I(x,y)}|}dt \leq C'(u_0)
$$
is satisfied for another constant $C'(u_0)>0$.
\end{lemma}

\begin{proof}
We fix any  $u_0 \in B_p^{\mathbb R}(\mathbb R)$.
By Proposition \ref{BpA}, there exists some $b \in L^\infty(\mathbb R)$
such that $\Vert u_0-b \Vert_* \leq C_{JN}/4$.
Then, $\Vert u-b \Vert_* \leq C_{JN}/2$ for every $u \in B_p(\mathbb R)$ 
with $\Vert u-u_0 \Vert_* \leq \Vert u-u_0 \Vert_{B_p} \leq C_{JN}/4$.
Applying the John--Nirenberg inequality as in (\ref{basicJN}), we have
\begin{equation*}
\begin{split}
\frac{1}{|I|}\int_{I}e^{|u(t)-b(t)-u_{I}+b_I|}dt
&\leq\frac{C_0 \Vert u-b \Vert_*}{C_{JN}- \Vert u-b \Vert_*}+1 \leq 2C_0+1. 
\end{split}
\end{equation*}
This implies that
$$
\frac{1}{|I|}\int_{I}e^{|u(t)-u_{I}|}dt \leq (2C_0+1)e^{2 \Vert b \Vert_\infty}.
$$
Then by (\ref{2cases}), the latter statement also follows.
\end{proof}

\section{The variant of the Beurling--Ahlfors extension for BMO}
Hereafter, we use the following convention.
The notation $A(u) \asymp B(u)$ concerning formulas of $u$ means that there is a constant $C \geq 1$ 
independent of $u$ such that $A(u)/C \leq B(u) \leq CA(u)$.
The notation $A(u) \lesssim B(u)$ means that there is such a constant $C$ that satisfies
$A(u) \leq CB(u)$. 
We also use a variant of this convention in the following case even if $C$ depends on $u$ in general.
When a function $u$ is in ${\rm BMO}(\mathbb R)$, 
the notation $A(u) \asymp B(u)$
means that there is a constant $C(u) \geq 1$ such that $A(u)/C(u) \leq B(u) \leq C(u)A(u)$,
where $C(u)$ is bounded whenever $\Vert u \Vert_*$ is bounded.

Beurling and Ahlfors \cite{BA} characterized the boundary value of a quasiconformal homeo\-morphism of
the upper half-plane $\mathbb U $ onto itself as a quasisymmetric homeomorphism $f$ of the real line $\mathbb R$.
Here, an increasing homeomorphism $f$ of $\mathbb R$ onto itself is {\it quasisymmetric} if there is a constant $\rho >1$,
which is called the doubling constant, such that
$|f(2I)| \leq \rho |f(I)|$ for any bounded interval $I \subset \mathbb R$, where $|\cdot|$ is the Lebesgue measure and
$2I$ denotes the interval of the same center as $I$
with $|2I|=2|I|$. 
Let 
$\phi(x)=\frac{1}{2} 1_{[-1,1]}(x)$ and $\psi(x)=\frac{r}{2} 1_{[-1,0]}(x)+\frac{-r}{2} 1_{[0,1]}(x)$
for some $r>0$, where $1_E$ denotes the characteristic function of $E \subset \mathbb R$. For any function $\varphi(x)$ 
on $\mathbb R$ and for $y>0$, we set $\varphi_y(x)=y^{-1} \varphi(y^{-1}x)$.
Then, for a quasisymmetric homeomorphism $f$, the Beurling--Ahlfors extension
$F(x,y)=(U(x,y),V(x,y))$ for $(x,y) \in \mathbb U$ is defined by the convolutions
$U(x,y)=(f \ast \phi_y)(x)$ and $V(x,y)=(f \ast \psi_y)(x)$.
Modification and variation to the Beurling--Ahlfors extension have been made by changing the functions
$\phi$ and $\psi$. 

For a complex-valued function $u$ on $\mathbb R$ 
such that $u \in \rm BMO(\mathbb R)$, we consider
a curve $\gamma_u=\gamma: \mathbb{R} \to \mathbb{C}$ given by
\begin{equation}\label{gamma}
\gamma(x) = \gamma(0) + \int_0^x e^{u(t)} dx.
\end{equation}
Let $\phi(x)=\frac{1}{\sqrt \pi}e^{-x^2}$ and $\psi(x)=\phi'(x)=-2x \phi(x)$.
Then, we extend $\gamma$ to the upper half-plane 
$\mathbb U$ setting a differentiable map $F_u=F: \mathbb{U} \to \mathbb{C}$ by 
\begin{equation}\label{F}
\begin{split}
&F(x, y) = U(x, y) + iV(x, y);\\
U(x,y)&=(\gamma \ast \phi_y)(x),\ V(x,y)=(\gamma \ast \psi_y)(x).
\end{split}
\end{equation}

The partial derivatives of
$U$ and $V$ can be represented as follows:
\begin{align*}
U_x(x,y)&=\frac{\partial U}{\partial x}=(e^u \ast \phi_y)(x);\\
V_x(x,y)&=\frac{\partial V}{\partial x}=(e^u \ast \psi_y)(x);\\
U_y(x,y)&=\frac{\partial U}{\partial y}=(\gamma \ast \frac{\partial \phi_y}{\partial y})(x)
=\frac{1}{2}(e^u \ast \psi_y)(x)=\frac{1}{2}V_x(x,y);\\
V_y(x,y)&=\frac{\partial V}{\partial y}=(\gamma \ast \frac{\partial \psi_y}{\partial y})(x)
=U_x(x,y)+\frac{y^2}{2}(e^u \ast (\phi_y)'')(x),
\end{align*} 
where we have used
\begin{align*}
\frac{\partial \phi_y}{\partial y}&=\frac{y}{2}\frac{\partial^2 \phi_y}{\partial x^2} =\frac{1}{2}(\psi_y)'; \\ 
\frac{\partial \psi_y}{\partial y}&=\frac{\partial \phi_y}{\partial x}+\frac{y^2}{2} \frac{\partial^3 \phi_y}{\partial x^3}
=(\phi_y)'+\frac{y^2}{2}(\phi_y)'''.
\end{align*}
In particular, each of 
$U_y(x,y)$, $V_x(x,y)$, and $(U_x-V_y)(x,y)$ can be represented by the convolution $(e^u \ast a_y)(x)$ explicitly
for a certain real-valued function $a \in C^{\infty}(\mathbb R)$
such that $\int_{\mathbb R} a(x)dx=0$, $|a(x)|$ is an even function, and $a(x)=O(x^2e^{-x^2})$ $(|x| \to \infty)$. 
For instance, 
$V_x(x,y)=(e^u \ast \psi_y)(x)$
for $\psi(x)=-\frac{2}{\sqrt \pi}xe^{-x^2}$. 

Next, we consider the complex derivatives 
\begin{align*}
F_{\bar z}(x,y) &= \frac{1}{2}(F_x + iF_y) = \frac{1}{2}\left((U_x - V_y) + iU_y + iV_x \right);\\
F_z(x,y) &= \frac{1}{2}(F_x - iF_y) = U_x+\frac{1}{2}\left(-(U_x- V_y) - iU_y+ iV_x  \right). 
\end{align*}
From these expressions, we can find two complex-valued functions $\alpha, \beta \in C^{\infty}(\mathbb R)$ 
independent of $u$ such that
\begin{equation}\label{alphabeta}
F_{\bar{z}}=e^u \ast \alpha_y(x), \quad F_{z}=e^u \ast \beta_y(x),
\end{equation}
and $\int_{\mathbb{R}}\alpha(x)dx = 0$, $\alpha(x)=O(x^2e^{-x^2})$ $(|x| \to \infty)$, 
$\int_{\mathbb{R}}\beta(x)dx = 1$, $\beta(x)=O(x^2e^{-x^2})$ $(|x| \to \infty)$.
In particular, we can assume that
$$
|\alpha(x)| \leq Ce^{-|x|}, \quad |\beta(x)| \leq Ce^{-|x|}
$$
for some constant $C>0$. 
We set $\mu_u(x,y)=F_{\bar z}/F_z$, and call it the complex dilatation of $F$ even though the map $F=F_u$ by $(\ref{F})$ is not necessarily quasiconformal. 

In the case where $u$ is a real-valued function such that $e^u$ is
an $A_\infty$-weight, the situation becomes simpler.
In this case, the extension $F:\mathbb U \to \mathbb U$ of $\gamma:\mathbb R \to \mathbb R$ is
the variant of the Beurling--Ahlfors extension by the heat kernel introduced in \cite{FKP}.
The following result is obtained in \cite[Theorems 3.4, 4.1]{WM-2}.

\begin{theorem}\label{FKP}
For a real-valued $u \in {\rm BMO}(\mathbb R)$ such that $e^u$ is an $\rm A_\infty$-weight,
the map $F_u$ induced by $u$ is a quasiconformal
diffeomorphism of $\mathbb U$ onto itself. Moreover, for the complex dilatation $\mu_u$ of $F_u$,
$\frac{1}{y}|\mu_u(x,y)|^2 dxdy$ is a Carleson measure on $\mathbb U$. If $u$ belongs to ${\rm VMO}(\mathbb R)$ in addition,
then this is a vanishing Carleson measure.
\end{theorem}

We say that a measure $\lambda(x,y)dxdy$ is a {\it Carleson measure} on $\mathbb U$ if
$$
\sup_{I \subset \mathbb R} \frac{1}{|I|}\int_0^{|I|}\!\! \int_I \lambda(x,y)dxdy 
$$
is bounded, where the supremum is taken over all bounded intervals $I$ in $\mathbb R$. Moreover, a Carleson measure
is {\it vanishing} if 
$$
\lim_{|I| \to 0} \frac{1}{|I|} \int_0^{|I|}\!\! \int_I \lambda(x,y)dxdy =0.
$$

Now, we add one more property on $F_u$ to this theorem. 
We say that a diffeomorphism $F(x,y)$ of $\mathbb U$ onto itself is bi-Lipschitz
with respect to the hyperbolic metric with constant $L \geq 1$ if
$$
\frac{1}{Ly} \leq \frac{| d_{(x,y)}F(v) |}{{\rm Im}\, F(x,y)} \leq  \frac{L}{y}
$$
for any $(x,y) \in \mathbb U$ and any unit tangent vector $v$ at $(x,y)$. The derivative satisfies
\begin{equation}\label{qcderivative}
\frac{1}{K}|F_z(x,y)| 
\leq | d_{(x,y)}F(v) | 
\leq K|F_z(x,y)|
\end{equation}
for the maximal dilatation $K=(1+\Vert \mu \Vert_\infty)/(1-\Vert \mu \Vert_\infty)$ with $\mu(x,y)=F_{\bar z}(x,y)/F_z(x,y)$.

\begin{proposition}\label{biLipproperty}
The quasiconformal
diffeomorphism $F_u$ in Theorem \ref{FKP} is
bi-Lipschitz with respect to the hyperbolic metric on $\mathbb U$. 
The maximal dilatation $K$ and the bi-Lipschitz constant $L$ of $F_u$ depend only on
the doubling constant $\rho$ of $e^u$. If $\Vert u \Vert_* <C_{JN}$,
then $\rho$ depends only on $\Vert u \Vert_*$.
\end{proposition}

\begin{proof}
We assume that $F=F_u$ is $K$-quasiconformal, and consider
\begin{align*}
|F_z(x,y)|^2&=\frac{1}{4}(U_x^2+U_y^2+V_x^2+V_y^2)+\frac{1}{2}(U_x V_y-U_y V_x)\\
&= \frac{1}{2}(U_x^2+U_y^2+V_x^2+V_y^2)-|F_{\bar z}(x,y)|^2\\
&\leq \frac{1}{2}(U_x^2+U_y^2+V_x^2+V_y^2).
\end{align*}
In the proof of \cite[Theorem 3.3]{WM-2}, we see that
\begin{equation*}\label{lip2}
U_x^2+U_y^2+V_x^2+V_y^2 \asymp U_x^2; \quad U_x V_y-U_y V_x \gtrsim U_x^2,
\end{equation*}
where the comparability is given in terms of the doubling constant $\rho$ of $e^u$.
Therefore, we have
\begin{equation}\label{Fz}
\begin{split}
|F_z(x,y)|\asymp U_x(x,y)&=(e^u \ast \phi_y)(x)=(\gamma \ast (\phi_y)')(x)\\
&=\frac{1}{y}(\gamma \ast \psi_y)(x)=\frac{1}{y}V(x,y)=\frac{1}{y}{\rm Im}\,F(x,y).
\end{split}
\end{equation}
Thus, the bi-Lipschitz constant $L$ of $F$ depends only on $\rho$ and $K$ by (\ref{qcderivative}) and (\ref{Fz}). Moreover,
again in the proof of \cite[Theorem 3.3]{WM-2}, we see that $K$ also depends only on $\rho$.

If $\Vert u \Vert_*<C_{JN}$, then $e^u \in A_2$ and the $A_2$-constant 
is given in terms of $\Vert u \Vert_*$ as is shown in Proposition \ref{C_0}.
Since the doubling constant for an $A_2$-weight depends only on the $A_2$-constant
(see \cite[Section 3]{CF}),
we see that $\rho$ depends only on $\Vert u \Vert_*$.
\end{proof}

We return to the general case where $u \in {\rm BMO}(\mathbb R)$ is complex-valued.  Parallel to \eqref{F}, the curve $\gamma_u = \gamma: \mathbb R \to \mathbb C$ given by \eqref{gamma}
is extendable also to the lower half-plane $\mathbb L$
 setting $F_u = F: \mathbb L \to \mathbb C$ by
\begin{equation}\label{lowerhalfplane}
\begin{split}
&F(x, y) = U(x, -y) - iV(x, -y);\\
U(x,y)&=(\gamma \ast \phi_y)(x),\ V(x,y)=(\gamma \ast \psi_y)(x). 
\end{split}
\end{equation}

In the rest of this section, 
we will show that under a small norm condition on $u \in {\rm BMO}(\mathbb R)$,
the variant of the Beurling--Ahlfors extension $F_u$ by the heat kernel yields
a quasiconformal homeomorphism of $\mathbb C$ (see Theorem \ref{smallcase} below).
This is an analogous result to that by Semmes \cite{Se} who
considered the case that the kernels $\phi$ and $\psi$ of the convolution are compactly supported.
In fact, the following two claims toward Theorem \ref{smallcase} are based on those in \cite[p.251]{Se}.

\begin{proposition}\label{Le}
Let $u$ and $\varphi$ be 
complex-valued functions on $\mathbb R$ 
such that $u \in \rm BMO(\mathbb R)$, $|\varphi(x)| \leq Ce^{-|x|}$ for some constant $C>0$, and $\int_{\mathbb{R}}\varphi(x)dx = 1$. 
Let $I(x,y) \subset \mathbb R$ be the interval $(x-y,x+y)$ for any $x \in \mathbb R$ and $y>0$.
Then, 
\begin{equation*}
|e^{\varphi_y \ast u(x)}| \asymp |e^{u_{I(x,y)}}|  
\end{equation*}
is satisfied.
\end{proposition}

\begin{proof}
We apply Proposition \ref{prep1} for $k=1$. Then, by setting $C_1=C(1)$, we see that
$$
|\varphi_y\ast u(x)-u_{I(x,y)}| \leq \int_{\mathbb{R}}|\varphi_y(x-t)||u(t)-u_{I(x,y)}|dt \leq C_1 \Vert u \Vert_*.
$$
This implies that
\begin{equation*}\label{ratio}
e^{-C_1\Vert u \Vert_*} \leq \left| e^{\varphi_y\ast u(x)-u_{I(x,y)}}\right| \leq e^{C_1\Vert u \Vert_*}, 
\end{equation*}
from which we have $\left|e^{\varphi_y\ast u(x)-u_{I(x,y)}}\right| \asymp 1$, that is, 
$|e^{\varphi_y \ast u(x)}| \asymp |e^{u_{I(x,y)}}|$.
\end{proof}

\begin{lemma}\label{semmes}
Let $u$ and $\varphi$ be 
complex-valued functions on $\mathbb R$ 
such that $u \in \rm BMO(\mathbb R)$, $|\varphi(x)| \leq Ce^{-|x|}$ for some constant $C>0$, 
and $\int_{\mathbb{R}}\varphi(x)dx = 1$. 
If $\Vert u \Vert_*$ is sufficiently small, then
\begin{equation*}
|e^{\varphi_y \ast u(x)}| 
\asymp |(\varphi_y\ast e^{u})(x)|
\end{equation*}
is satisfied.
\end{lemma}

\begin{proof}
By Proposition \ref{Le}, it suffices to show that
$|e^{u_{I(x,y)}}| \asymp |(\varphi_y\ast e^{u})(x)|$.
We will estimate
\begin{equation*}\label{minus1}
\left|\frac{(\varphi_y\ast e^{u})(x)}{e^{u_{I(x,y)}}}-1 \right|=
|\varphi_y\ast (e^{u(\cdot)-u_{I(x,y)}}-1)(x)|
\end{equation*}
from above when $\Vert u \Vert_*<C_{JN}/2$. 
This can be done by
\begin{equation}\label{correct}
\begin{split}
&\quad|\varphi_y\ast (e^{u(\cdot)-u_{I(x,y)}}-1)(x)|\\
&\leq \int_{\mathbb R} |\varphi_y(x-t)||e^{u(t)-u_{I(x,y)}}-1|dt \\
&\leq \int_{\mathbb R} \left| \varphi_y (x-t)\right|
|u(t)-u_{I(x,y)}|e^{|u(t)-u_{I(x,y)}|} dt\\
&\leq \left(\int_{\mathbb R} |\varphi_y (x-t)|
|u(t)-u_{I(x,y)}|^2dt \right)^{1/2}\left(\int_{\mathbb R} |\varphi_y (x-t)|e^{2|u(t)-u_{I(x,y)}|}dt \right)^{1/2}, 
\end{split}
\end{equation}
where we have used an inequality $|e^z-1| \leq |z|e^{|z|}$.

Here, Proposition \ref{prep1} implies that the first factor of the last line in (\ref{correct}) is bounded by 
a multiple of $\Vert u \Vert_*$, and Proposition \ref{prep2} implies that the second factor is bounded if 
$\Vert 2u \Vert_*<C_{JN}$.
Therefore, we have
$$
\left|\frac{(\varphi_y\ast e^{u})(x)}{e^{u_{I(x,y)}}}-1 \right| \lesssim \Vert u \Vert_*
$$
in this case, and thus  
$|e^{\varphi_y \ast u(x)}| 
\asymp |(\varphi_y\ast e^{u})(x)|$ when $\Vert u \Vert_*$ is sufficiently small.
\end{proof}

From these two claims, we see that the supremum norm of the complex dilatation is dominated by the BMO norm. In the following Proposition \ref{mu<BMO}, we only consider the upper half-plane case. The lower half-plane case can be treated similarly. 

\begin{proposition}\label{mu<BMO}
For a complex-valued function $u$ on $\mathbb R$ with $u \in \rm BMO(\mathbb R)$,
let $\mu_u$ be the complex dilatation of
the map $F_u$ on $\mathbb U$ induced by $u$ as above. If $\Vert u \Vert_*$ is sufficiently small,
then $\Vert \mu_u \Vert_\infty \lesssim \Vert u \Vert_*$.
\end{proposition}

\begin{proof}
It follows from Proposition \ref{Le} and Lemma \ref{semmes} that 
\begin{equation*}\label{mu}
|\mu_u(x,y)| = \frac{|\alpha_y\ast e^u(x)|}{| \beta_y\ast e^u(x)|}
\asymp  \frac{|\alpha_y\ast e^u(x)|}{| e^{\beta_y\ast u(x)}|}\\
\asymp  \frac{|\alpha_y\ast e^u(x)|}{| e^{u_{I(x,y)}}|}\\
=|\alpha_y\ast e^{u - u_{I(x,y)}}(x)|.
\end{equation*}
Then, by $\int_{\mathbb{R}}\alpha(x)dx = 0$, which implies $\int_{\mathbb{R}}\alpha_y(x)dx = 0$,  
and by $|e^z - 1| \leq |z|e^{|z|}$, 
we have 
\begin{equation}\label{mu}
\begin{split}
|\mu_u(x,y)| &\asymp |\alpha_y\ast (e^{u - u_{I(x,y)}} - 1)(x)|\\
&\leq\int_{\mathbb{R}}|\alpha_y(x-t)| |u(t) - u_{I(x,y)}| e^{|u(t) - u_{I(x,y)}|} dt.
\end{split}
\end{equation}
By (\ref{correct}) applying to $\alpha=\varphi$, the last line of (\ref{mu})
is bounded by some multiple of $\Vert u \Vert_*$.
This implies that $\Vert \mu_u \Vert_\infty \lesssim \Vert u \Vert_*$.
\end{proof}

Finally, aiming at the goal of this section,
we prove a certain approximation of BMO functions.
This was stated in \cite[p.251]{Se} without proof in order to prove that $F_u$ is quasiconformal
in its situation.

\begin{lemma}\label{approximation}
For any $u \in {\rm BMO}(\mathbb R)$,
there exist a sequence $\{u_j\} \subset {\rm BMO}(\mathbb R)$ and constants $a \in \mathbb C$ and $C>0$ such that 
$u_j-a$ is continuous and compactly supported with
$\Vert u_j \Vert_\ast \leq C \Vert u \Vert_\ast$ for all $j \in \mathbb N$ and 
$u_j$ converges to $u$ locally 
in $L^p$ for $1 \leq p <\infty$, i.e., $u_j \to u$ in $L^p(I)$ for any bounded interval $I \subset \mathbb R$.
Moreover, if $\Vert u \Vert_\ast$ is sufficiently small in addition, then
$e^{u_j}$ converges to
$e^u$ locally in $L^1$.
\end{lemma}

\begin{proof}
Any function $u \in {\rm BMO}(\mathbb R)$ can be written as
$u=\varphi+ H(\psi)+a$, where $\varphi, \psi \in L^\infty(\mathbb R)$ satisfy
$\Vert \varphi \Vert_\infty \lesssim \Vert u \Vert_\ast$, $\Vert \psi \Vert_\infty \lesssim \Vert u \Vert_\ast$,
$a$ is a constant, and $H$ is the Hilbert transformation for $L^\infty(\mathbb R)$ defined by
$$
H(\psi)(x)=\lim_{\varepsilon \to 0} \frac{1}{\pi} \int_{|x-t|>\varepsilon} \left(\frac{1}{x-t}+\frac{t}{1+t^2}\right) \psi(t)dt.
$$
See \cite[Corollary VI.4.5]{Ga}. 
Moreover, let
$$
b(\psi)(x)=\frac{1}{\pi} \int_{|x-t|\geq 1} \left(\frac{1}{x-t}+\frac{t}{1+t^2}\right) \psi(t)dt,
$$
which satisfies $b(\psi) \in L^\infty(\mathbb R)$ with $\Vert b(\psi) \Vert_\infty \lesssim \Vert \psi \Vert_\infty$. Then,
$
u(x)=\varphi(x) +b(\psi)(x)+H'(\psi)(x)+a
$ 
for
$$
H'(\psi)(x)=\lim_{\varepsilon \to 0} \frac{1}{\pi} \int_{1>|x-t|>\varepsilon} \left(\frac{1}{x-t}+\frac{t}{1+t^2}\right) \psi(t)dt.
$$

We define a function 
\begin{equation}\label{mollifier}
\eta(x)=
\begin{cases}
c \exp (\frac{1}{x^2-1}) & (-1<x<1)\\ 
0 & {\rm otherwise},  
\end{cases}
\end{equation}
where
$c$ is a constant satisfying $\int_{\mathbb R} \eta(x)dx=1$. Then, we consider the mollifier
$\eta_{1/j}(x)=j \eta(jx)$ for every $j \in \mathbb N$, and define
\begin{equation*}
\begin{split}
u_j(x)&=\eta_{1/j} \ast \left(\varphi1_{[-j,j]}+b(\psi)1_{[-j,j]}+H'(\psi1_{[-j,j]}) \right)(x)+a\\
&=\eta_{1/j} \ast \left(\varphi1_{[-j,j]}+b(\psi)1_{[-j,j]}-b(\psi1_{[-j,j]})+H(\psi1_{[-j,j]}) \right)(x)+a\\
&=\eta_{1/j} \ast \left(\varphi1_{[-j,j]}+b(\psi)1_{[-j,j]}-b(\psi1_{[-j,j]})\right)(x)+H(\eta_{1/j} \ast(\psi 1_{[-j,j]}))(x)+a.
\end{split}
\end{equation*}
For the third equality above, we used a fact that the convolution by the mollifier and the Hilbert transformation commute
(see \cite[Lemma 6]{To}).
Then, each $u_j-a$ is smooth and compactly supported.

By a property of the mollifier, we have
\begin{equation*}
\begin{split}
\Vert \eta_{1/j} \ast (\varphi1_{[-j,j]}) \Vert_* &\leq \Vert \eta_{1/j} \ast (\varphi1_{[-j,j]}) \Vert_\infty \leq \Vert \varphi \Vert_\infty \lesssim \Vert u \Vert_\ast;\\
\Vert \eta_{1/j} \ast (b(\psi)1_{[-j,j]}) \Vert_* &\leq \Vert \eta_{1/j} \ast (b(\psi)1_{[-j,j]}) \Vert_\infty \leq \Vert b(\psi) 1_{[-j,j]}\Vert_\infty \lesssim \Vert \psi \Vert_\infty 
\lesssim \Vert u \Vert_\ast;\\
\Vert \eta_{1/j} \ast (b(\psi 1_{[-j,j]})) \Vert_* &\leq \Vert \eta_{1/j} \ast (b(\psi 1_{[-j,j]})) \Vert_\infty \leq \Vert b(\psi 1_{[-j,j]}) \Vert_\infty \lesssim \Vert \psi \Vert_\infty 
\lesssim \Vert u \Vert_\ast.
\end{split}
\end{equation*}
Moreover, for the Hilbert transformation $H$, we know that
$\Vert H(\psi) \Vert_* \lesssim \Vert \psi \Vert_\infty$ (see \cite[Theorem VI.1.5]{Ga}). Hence,
\begin{equation*}
\begin{split}
\Vert H(\eta_{1/j} \ast(\psi1_{[-j,j]})) \Vert_* 
&\lesssim \Vert \eta_{1/j} \ast(\psi1_{[-j,j]}) \Vert_\infty \leq \Vert \psi \Vert_\infty \lesssim \Vert u \Vert_\ast.
\end{split}
\end{equation*}
Thus, we can conclude that there is some constant $C>0$ such that $\Vert u_j \Vert_* \leq C \Vert u \Vert_*$
for all $j \in \mathbb N$.

Let $I \subset \mathbb R$ be an arbitrary bounded interval. Then, for all sufficiently large $j$, 
we have $\varphi1_{[-j,j]}=\varphi$, $b1_{[-j,j]}=b$, and $H'(\psi1_{[-j,j]})=H'(\psi)$ on $I$.
Therefore, by another property of the mollifier, we see that $u_j$ converges to $u$ in $L^p(I)$
for $1 \leq p < \infty$.

Finally, we show that $e^{u_j} \to e^u$ in $L^1(I)$ when $\Vert u \Vert_*$ is sufficiently small.
For simplicity, we may assume that $u_I=0$, for otherwise,
we only have to put the constant $u_I$ at appropriate places. Since $u_j \to u$ in $L^1(I)$, we have $(u_j)_I \to 0$
as $j \to \infty$. Hence, $|(u_j)_I| \leq \delta$ for some $\delta>0$.
As before, we have
\begin{equation*}
\begin{split}
&\quad\ \frac{1}{|I|}\int_I \left| e^{u(t)}-e^{u_j(t)} \right| dt 
\leq \frac{1}{|I|}\int_I e^{|u(t)|} e^{|u(t)-u_j(t)|}|u(t)-u_j(t)|dt\\
&\leq \frac{e^\delta}{|I|} \int_I e^{2|u(t)-u_I|}e^{|u_j(t)-(u_j)_I|}|u(t)-u_j(t)|dt\\
&\leq e^\delta \left(\frac{1} {|I|} \int_I e^{4|u(t)-u_I|}dt \right)^{1/2}
\left(\frac{1} {|I|} \int_I e^{4|u_j(t)-(u_j)_I|}dt \right)^{1/4}\left(\frac{1} {|I|} \int_I |u(t)-u_j(t)|^4dt \right)^{1/4}.
\end{split}
\end{equation*}
By the John--Nirenberg inequality as used in (\ref{basicJN}), 
the first integral factor in the last line is bounded in terms of $\Vert u \Vert_*$ 
and the second integral factor is bounded in terms of $\Vert u_j \Vert_* \leq C\Vert u \Vert_*$ if $\Vert u \Vert_*$
is sufficiently small. Since $u_j \to u$ in $L^p(I)$, this proves that $e^{u_j} \to e^u$ in $L^1(I)$.
\end{proof}

After these preparations, we obtain the following theorem on
the variant of the Beurling--Ahlfors extension $F_u$ by the heat kernel for BMO functions $u$. 
This asserts that $F_u$ is quasiconformal if $\Vert u \Vert_\ast$ is small.
Combining this with Theorem \ref{FKP}, we see that $F_u$ is quasiconformal if $u$ is either
real-valued with $e^u \in A_\infty$ or of small BMO norm. In the next section, we consider this extension
in the case where $u$ is in the $p$-Besov space.

\begin{theorem}\label{smallcase}
The map $F_u$ extends continuously to
a quasiconformal homeomorphism of $\mathbb C$ onto itself with $F_u|_{\mathbb R}=\gamma_u$
if $\Vert u \Vert_\ast$ is sufficiently small.
\end{theorem}

\begin{proof}
By Proposition \ref{mu<BMO}, we may assume that $\Vert \mu_u \Vert_\infty$ is also sufficiently small, 
and in particular $\Vert \mu_u \Vert_\infty<1$.
By the property of heat kernel, we see that 
$U(x,y) \to \gamma_u(x)$ and $V(x,y) \to 0$ as $y \to 0$. This shows that $F=F_u$ extends continuously 
to $\gamma_u$ on $\mathbb R$, and then $F$ is continuous on $\mathbb C$. 

Suppose first that $u-a$ is continuous and has a compact support for some $a \in \mathbb C$. 
Then, $F_{\bar z}$ and $F_z$ are continuous off $\mathbb R$. Noting that 
$F_{\bar z} = e^u\ast\alpha_y(x)$, $F_z = e^u\ast\beta_y(x)$ and $\int_{\mathbb R} \alpha(x) dx = 0$, $\int_{\mathbb R} \beta (x) dx = 1$, we conclude by the Lebesgue dominated convergence theorem that $F_{\bar z}(x,y) \to 0$ and $F_z(x,y) \to \gamma'(x)$ as $y \to 0$. Thus, $F$ is continuously differentiable on $\mathbb C$. 
Since $u-a$ has a compact support, $\gamma(x) = e^a x + O(1)$ at $\infty$, and then $F(z) = e^a z + O(1)$ at $\infty$, which implies that $F(z) \to \infty$ as $z \to \infty$. Moreover, since $|F_z(x, y)| \asymp |e^{\beta_y\ast u(x)}|$  by Lemma \ref{semmes}, which is never $0$, 
the Jacobian determinant of $F$ is positive everywhere. This implies that
$F$ is locally homeomorphic on $\mathbb{U}$. 
The same is true for the map $F$ defined on $\mathbb L$.
Then, a topological argument deduces that $F$ extends continuously to an orientation-preserving global homeomorphism of 
$\mathbb C$ onto itself. 
By $\Vert \mu_u \Vert_\infty<1$, we see that $F$ is quasiconformal. 

Consider now the general case. Given $u \in \rm BMO(\mathbb R)$ with a small norm, 
Lemma \ref{approximation} implies that
there is a continuous and compactly supported sequence $\{u_j-a\} \subset {\rm BMO}(\mathbb R)$ such that 
$\Vert u_j \Vert_* \leq C \Vert u \Vert_*$, $u_j \to u$ locally in $L^1$, and $e^{u_j} \to e^{u}$ locally in $L^1$.
We assume that $\gamma_j(0) = \gamma(0)$ for all $j$.
The corresponding $F_{u_j}$ for $u_j$ is a quasiconformal homeomorphism of $\mathbb C$ by the above arguments.
Since $\gamma_{u_j} \to \gamma$ uniformly on compact sets by 
the condition that $e^{u_j} \to e^{u}$ locally in $L^1$, so does $F_{u_j} \to F$.
Since $\Vert u_j \Vert_*$ is also sufficiently small
by $\Vert u_j \Vert_* \leq C \Vert u \Vert_*$, we can regard that all $\Vert \mu_{u_j} \Vert_\infty$ are
uniformly bounded by a constant less than $1$. Then, passing to a subsequence if necessary,
$F_{u_j}$ converges to a quasiconformal homeomorphism of $\mathbb C$, which coincides with
$F_u$. Thus, we see that $F_u$ is quasiconformal.
\end{proof}

\section{Quasiconformal extension of curves}
In this section, we establish the main result of this paper (Theorem \ref{qcU}, which is the precise version of Theorem \ref{thm5}). The key point is to consider the quasiconformal extension of a curve $\gamma_u:\mathbb R \to \mathbb C$ produced by
$u \in B_p(\mathbb R)$ by applying the variant of the Beurling--Ahlfors extension by the heat kernel. 
Then, several results concerning the $p$-Weil--Petersson Teichm\"uller space follow this as mentioned in Section 1.

Let us start with some notations. Let $M(\mathbb U)$ denote the open unit ball of the Banach space $L^{\infty}(\mathbb U)$
of all essentially bounded measurable functions on $\mathbb U$. An element in $M(\mathbb U)$ is called a
Beltrami coefficient. 
Let $L_0(\mathbb U)$ denote a subspace of $L^{\infty}(\mathbb U)$ consisting 
of all elements $\mu$ vanishing at the boundary, that is,
$$
\lim_{t \to 0}\, {\rm ess.}\!\!\!\!\!\!\!\!\sup_{0<y<t\qquad} \!\!\!\!\!\!\!\! |\mu(x,y)| = 0. 
$$
Then, we define the subset of all Beltrami coefficients vanishing at the boundary by
$M_0(\mathbb U) =  M(\mathbb U) \cap L_0(\mathbb U)$. 
For $\mu \in L^\infty(\mathbb U)$, we define the $p$-integrable norm 
with respect to the hyperbolic metric by
\begin{equation*}
\Vert \mu \Vert_p = \left(\iint_{\mathbb U}|\mu(x,y)|^p \frac{dxdy}{y^2} \right)^{\frac{1}{p}}. 
\end{equation*}
Then, we introduce a new norm $\Vert \mu \Vert_{\infty}+\Vert \mu \Vert_{p}$ for $\mu$.
Let $\mathcal{L}_p(\mathbb U)$ denote a subspace of $L^{\infty}(\mathbb U)$ consisting of all elements $\mu$ with 
$\Vert \mu \Vert_{\infty}+\Vert \mu \Vert_{p}<\infty$, which is a Banach space.
Moreover, we set $\mathcal{M}_p(\mathbb U) =  M(\mathbb U) \cap \mathcal{L}_p(\mathbb U)$. 

As before, $\mu_u$ denotes the complex dilatation of $F_u$
given by (\ref{F}) and (\ref{alphabeta}), but we take
$u \in B_p(\mathbb R)$ in the present setting. The following three claims are concerning
the boundedness of the norms of $\mu_u$.

\begin{proposition}\label{bounded}
For any  $u_0 \in B^{\mathbb R}_p(\mathbb R)$, there are $\delta>0$ and $M>0$ such that
if $\tilde u \in B_p(\mathbb R)$ satisfies
$\Vert \tilde u-u_0 \Vert_{B_p} <\delta$, then $\Vert \mu_{\tilde u} \Vert_{\infty}+\Vert \mu_{\tilde u} \Vert_{p} \leq M$.
\end{proposition}

\begin{proof}
For any $\tilde u \in B_p(\mathbb R)$, 
we set $\tilde u=u+iv$ where $u$ and $v$ are real-valued.
Fixing $I(x,y)=(x-y,x+y) \subset \mathbb R$ for $(x,y) \in \mathbb U$,
we consider
\begin{equation}\label{mu-fraction}
|\mu_{\tilde u}(x,y)|=\frac{|\alpha_y\ast e^{\tilde u}(x)|}{|\beta_y\ast e^{\tilde u}(x)|}
=\frac{|\alpha_y\ast e^{\tilde u- \tilde u_{I(x,y)}}(x)|}{|\beta_y\ast e^{\tilde u-\tilde u_{I(x,y)}}(x)|}.
\end{equation}
The denominator is estimated from below as
\begin{equation*}
\begin{split}
|\beta_y\ast e^{\tilde u-\tilde u_{I(x,y)}}(x)|&=|\beta_y\ast(e^{u-u_{I(x,y)}}(e^{i(v-v_{I(x,y)})}-1))
+\beta_y\ast e^{u-u_{I(x,y)}}|\\
&\geq |\beta_y\ast e^{u-u_{I(x,y)}}|-|\beta_y\ast(e^{u-u_{I(x,y)}}(e^{i(v-v_{I(x,y)})}-1))|.
\end{split}
\end{equation*}
Here, since $u$ is real-valued, we see as in the proof of Proposition \ref{biLipproperty} that 
$$
|\beta_y\ast e^{u-u_{I(x,y)}}| \asymp 
|\phi_y\ast e^{u-u_{I(x,y)}}| 
$$
for $\phi(x)=\frac{1}{\sqrt \pi}e^{-x^2}$. 
The Jensen inequality implies that
\begin{equation}\label{s2}
\begin{split}
\left|(\phi_y\ast e^{u-u_{I(x,y)}})(x)\right| & \geq \int_{\mathbb{R}}\phi_y(x-t) e^{-|u(t)-u_{I(x,y)}|} dt\\
& \geq \frac{2}{\sqrt \pi e}\left(\frac{1}{2y}\int_{|x-t|<y} e^{-|u(t)-u_{I(x,y)}|} dt \right)\\ 
& \geq \frac{2}{\sqrt \pi e}\exp (-\Vert u \Vert_*) 
\gtrsim  1.
\end{split}
\end{equation}
On the contrary, the Cauchy--Schwarz inequality and
$|e^{ix}-1| \leq |x|$
yield that
\begin{equation}\label{denomi}
\begin{split}
&\quad |\beta_y\ast(e^{u-u_{I(x,y)}}(e^{i(v-v_{I(x,y)})}-1))(x)| \\
&\leq \int_{\mathbb R} |\beta_y(x-t)|e^{u(t)-u_{I(x,y)}}|e^{i(v(t)-v_{I(x,y)})}-1|dt \\
&\leq \left(\int_{\mathbb R} |\beta_y(x-t)|e^{2(u(t)-u_{I(x,y)})}dt \right)^{1/2}
\left(\int_{\mathbb R} |\beta_y(x-t)||v(t)-v_{I(x,y)}|^2 dt \right)^{1/2}.
\end{split}
\end{equation}
By Lemma \ref{nearreal}, the first factor of the last line of (\ref{denomi}) is locally bounded.
By Proposition \ref{prep1}, the second factor is bounded by a multiple of $\Vert v \Vert_* \leq \Vert v \Vert_{B_p}$.
Thus, the denominator in the fraction of (\ref{mu-fraction}) is
bounded away from $0$ if $\Vert v \Vert_{B_p}$ is sufficiently small depending on 
where $u$ moves.
In particular, there is some $\delta>0$ with $\delta \leq C_{JN}/8$ such that if $\Vert u-u_0 \Vert_{B_p} < \delta/2$ and
if $\Vert v \Vert_{B_p} < \delta/2$, then the denominator is uniformly bounded away from $0$. 
This in particular shows that there is some $C>0$ such that if $\Vert \tilde u-u_0 \Vert_{B_p} <\delta$ then
\begin{equation}\label{onlyalpha}
|\mu_{\tilde u}(x,y)| \leq C |\alpha_y\ast e^{\tilde u-\tilde u_{I(x,y)}}(x)|.
\end{equation}

The right hand $|\alpha_y\ast e^{\tilde u-\tilde u_{I(x,y)}}(x)|$ coincides with
$|\alpha_y\ast (e^{\tilde u-\tilde u_{I(x,y)}}-1)(x)|$ by $\int_{\mathbb R} \alpha(x)dx=0$. Then, similarly to the above estimate,
we have
\begin{equation}\label{numer}
\begin{split}
&\quad |\alpha_y\ast (e^{\tilde u-\tilde u_{I(x,y)}}-1)(x)|\\
&\leq \int_{\mathbb R} |\alpha_y(x-t)||\tilde u(t)-\tilde u_{I(x,y)}|e^{|\tilde u(t)-\tilde u_{I(x,y)}|}dt\\
&\leq \left(\int_{\mathbb R} |\alpha_y(x-t)||\tilde u(t)-\tilde u_{I(x,y)}|^2dt \right)^{1/2}
\left(\int_{\mathbb R} |\alpha_y(x-t)|e^{2|\tilde u(t)-\tilde u_{I(x,y)}|}dt \right)^{1/2}.
\end{split}
\end{equation}
Again by Proposition \ref{prep1} and Lemma \ref{nearreal}, this is bounded if
$\Vert \tilde u-u_0 \Vert_{B_p} \leq C_{JN}/8$.
By \eqref{onlyalpha}, we see that $\Vert \mu_{\tilde u} \Vert_\infty$ is bounded if
$\Vert \tilde u-u_0 \Vert_{B_p} <\delta$. By \eqref{onlyalpha} again,  the boundedness of $\Vert \mu_{\tilde u} \Vert_p$ follows from  
Lemma \ref{mainlemma} below. Thus, the proof is completed.
\end{proof}

\begin{lemma}\label{mainlemma}
Suppose that the complex dilatation $\mu$ on $\mathbb U$ is given so that
$$
|\mu(x,y)| \lesssim |\alpha_y\ast e^{u - u_{I(x,y)}} (x)|
$$
for $\alpha \in C^\infty(\mathbb R)$ with $\int_{\mathbb{R}}\alpha(x)dx = 0$ and $|\alpha(x)| \leq Ce^{-|x|}$ and
for $u \in {\rm BMO}(\mathbb R)$. 
If $u \in B_p(\mathbb R)$ is within norm $C_{JN}/(4q)$ from $B^{\mathbb R}_p(\mathbb R)$ for $1/p+1/q=1$,
then 
\begin{equation*}
\Vert \mu \Vert_p^p=\iint_{\mathbb{U}}\frac{|\mu(x,y)|^p}{y^2} dxdy \leq C_p(u)  \Vert u \Vert_{B_p}^{p}, 
\end{equation*}
where the constant $C_p(u)>0$ is locally bounded in the neighborhood of $B^{\mathbb R}_p(\mathbb R)$.
\end{lemma}

\begin{proof}
By $\int_{\mathbb{R}}\alpha(x)dx = 0$ and the inequality $|e^z-1| \leq |z|e^{|z|}$, we have  
\begin{equation}\label{setup}
\begin{split}
|\mu(x,y)|^p &\lesssim |\alpha_y\ast (e^{u - u_{I(x,y)}} - 1)(x)|^p\\
&\leq\left(\int_{\mathbb{R}}|\alpha_y(x-t)| |u(t) - u_{I(x,y)}| e^{|u(t) - u_{I(x,y)}|} dt\right)^p\\
&\leq \left(\int_{\mathbb{R}}|\alpha_y(x-t)| |u(t) - u_{I(x,y)}|^p dt \right)  
\left(\int_{\mathbb{R}}|\alpha_y(x-t)| e^{q|u(t) - u_{I(x,y)}|} dt \right)^{p/q}.\\
\end{split}
\end{equation}
Since $u \in B_p(\mathbb R)$ is in the small neighborhood of 
$B_p^{\mathbb R}(\mathbb R)$, Lemma \ref{nearreal} implies that the second factor of the last line above
is bounded by a locally bounded constant $C(u)>0$. Thus, we have only to consider the first factor.

As in the proof of Proposition \ref{prep1},
we decompose the first factor into 
\begin{equation}\label{factor}
\begin{split}
&\quad\int_{\mathbb{R}}|\alpha_y(x-t)| |u(t) - u_{I(x,y)}|^p dt \\
& \leq \frac{2C}{|I_0|} \int_{I_0} |u(t) - u_{I_0}|^p dt
+ \sum_{n=0}^{\infty} \frac{2^{n+2}C}{e^{2^n}|I_{n+1}|}\int_{I_{n+1}} |u(t) - u_{I_0}|^p dt\\
& \leq C \sum_{n=0}^{\infty} \frac{2^{n+2}}{e^{2^{n-1}}|I_{n}|}\int_{I_{n}} |u(t) - u_{I_0}|^p dt
=C\sum_{n=0}^{\infty} \frac{2^{n+2}}{e^{2^{n-1}}}A_n,
\end{split}
\end{equation}
where we set
$$
A_{n}=\frac{1}{|I_{n}|}\int_{I_{n}} |u(t) - u_{I_0}|^p dt
$$
for $I_n=(x+2^ny,x-2^ny)$
for every integer $n \geq 0$. 
Moreover, by using 
 the H\"older inequality, we compute 
\begin{equation*}
\begin{split}
A_{n}& 
=\frac{1}{|I_{n}|}\int_{I_{n}} \left|u(t) - \frac{1}{|I_0|} \int_{I_0}u(s) ds  \right|^p dt
=\frac{1}{|I_{n}|}\int_{I_{n}}\left| \frac{1}{|I_0|}\int_{I_0}(u(t)-u(s)) ds \right|^p dt\\
&
\leq \frac{1}{|I_{n}||I_0|}\int_{I_{n}}\int_{I_0} |u(t) - u(s)|^p ds dt.\\
\end{split}
\end{equation*}
By the translation and by $(a + b)^{p} \leq 2^{p-1}(a^{p} + b^{p})$ for $a,b \geq 0$, 
the last line above is further estimated by
\begin{equation*}
\begin{split}
&\quad \frac{1}{|I_{n}||I_0|}\int_{I_{n}}\int_{I_0} |u(t) - u(s)|^p ds dt\\
&\leq \frac{2^{p-1}}{2^{n+2}y^2} \int_{-2^{n}y}^{2^{n}y} \int_{-y}^{y} (|u(x+t) - u(x) |^p + |u(x + s) - u(x) |^p )ds dt\\
& = 2^{p-1}\left(\frac{1}{2^{n+1}y}\int_{-2^{n}y}^{2^{n}y} |u(x+t) - u(x)|^p dt + \frac{1}{2y}\int_{-y}^{y} |u(x+s) - u(x)|^p ds \right)\\
& = 2^{p-1}(L_{n} + L_{0}),
\end{split}
\end{equation*}
where we set 
$$
L_{n}=\frac{1}{2^{n+1}y}\int_{-2^{n}y}^{2^{n}y} |u(x+t) - u(x)|^p dt
$$
for every integer $n \geq 0$. 

By the substitution of all the above computation into (\ref{setup}), we obtain that
\begin{equation}\label{p}
\begin{split}
|\mu(x,y)|^p  \lesssim C(u)C\sum_{n=0}^{\infty} \frac{2^{n+2}}{e^{2^{n-1}}}A_n
\leq C_p(u) \sum_{n=0}^{\infty}\frac{1}{e^{2^n}} (L_n+L_0)
\end{split}
\end{equation}
for some constant $C_p(u)>0$, which is locally bounded in the neighborhood of $B^{\mathbb R}_p(\mathbb R)$.
Moreover, the integral of each term $L_n$ over $\mathbb U$ is explicitly given as follows:
\begin{equation}\label{compute}
\begin{split}
\iint_{\mathbb{U}} \frac{L_{n}}{y^2} dxdy  
&= \frac{1}{2^{n+1}}\iint_{\mathbb{U}} dxdy\int_{-2^{n}y}^{2^{n}y} \frac{|u(x+t) - u(x)|^p}{y^3} dt \\
& =\frac{1}{2^{n+1}}\int_{-\infty}^{+\infty} dx \int_{0}^{+\infty} \frac{dy}{y^3}  \int_0^{2^{n}y} (|u(x+t) - u(x)|^p + |u(x-t) - u(x)|^p) dt\\
& = \frac{1}{2^{n+1}}\int_{-\infty}^{+\infty} dx \int_{0}^{+\infty}  (|u(x+t) - u(x)|^p + |u(x-t) - u(x)|^p) dt \int_{2^{-n}t}^{+\infty}\frac{dy}{y^3} \\
& = 2^{n-2} \int_{-\infty}^{+\infty} dx \int_0^{+\infty}\frac{|u(x+t) - u(x)|^p + |u(x-t) - u(x)|^p}{t^2} dt\\
& = 2^{n-2} \int_{-\infty}^{+\infty} dx \int_{-\infty}^{+\infty}\frac{|u(x+t) - u(x)|^p}{t^2} dt
= 2^{n-2} \Vert u \Vert_{B_p}^{p}.
\end{split}
\end{equation}
Therefore,
\begin{equation*}
\iint_{\mathbb{U}}\frac{|\mu(z)|^p}{y^2} dxdy 
\leq C_p(u) \sum_{n=0}^\infty \frac{1}{e^{2^n}}\left(2^{n-2}+2^{-2} \right) \Vert u \Vert_{B_p}^{p}
\leq C_p(u) \Vert u \Vert_{B_p}^{p}, 
\end{equation*}
which proves the statement of the lemma.
\end{proof}

\begin{proposition}\label{M_0}
Under the same circumstances as in Lemma \ref{mainlemma}, 
$$
\lim_{t \to 0} \sup_{0<y<t} |\mu(x,y)|=0,
$$
that is, $\mu \in L_0(\mathbb U)$. 
\end{proposition}

\begin{proof}
By (\ref{setup}) and (\ref{factor}) for $p=2$, we have
\begin{equation}\label{factor2}
|\mu(x,y)|^2 \lesssim C(u)C\sum_{n=0}^{\infty} \frac{2^{n+2}}{e^{2^{n-1}}|I_{n}|}\int_{I_{n}} |u(t) - u_{I_0}|^2 dt,
\end{equation}
and by (\ref{k!}) and (\ref{2n^k}), we have
\begin{equation}\label{together}
\begin{split}
&\quad \frac{1}{|I_{n}|}\int_{I_{n}} |u(t) - u_{I_0}|^2 dt\\
&\leq
\frac{2}{|I_{n}|}\int_{I_{n}}|u(t) - u_{I_{n}}|^2dt+\frac{2}{|I_{n}|}\int_{I_{n}}|u_{I_{n}} - u_{I_0}|^2dt 
\leq 2\left(\frac{2C_0}{C_{JN}^2}+4n^2 \right)\Vert u \Vert_*^2.
\end{split}
\end{equation}
This implies that the infinite series in (\ref{factor2}) converges.
Hence, for an arbitrarily given $\varepsilon>0$, we can choose $N \in \mathbb N$ such that
$$
C(u)C\sum_{n=N+1}^{\infty} \frac{2^{n+2}}{e^{2^{n-1}}|I_{n}|}\int_{I_{n}} |u(t) - u_{I_0}|^2 dt<\varepsilon^2.
$$

Since $u \in {\rm VMO}(\mathbb R)$ by Proposition \ref{BMOnorm}, there is $\delta>0$ such that
for any interval $J \subset \mathbb R$ with $|J| \leq \delta$, we have
\begin{equation}\label{smallJ}
\frac{1}{|J|} \int_J |u(t)-u_J|dt \leq \varepsilon.
\end{equation}
If $2^{N+1}y \leq \delta$, then $|I_n| \leq \delta$ for $0 \leq n \leq N$, and hence inequality (\ref{smallJ}) is valid for
$J=I_n$ $(0 \leq n \leq N)$. Then, we apply the John--Nirenberg inequality (\ref{JN}) restricted to these small intervals.
By using (\ref{together}), we can estimate
$$
C(u)C\sum_{n=0}^{N} \frac{2^{n+2}}{e^{2^{n-1}}|I_{n}|}\int_{I_{n}} |u(t) - u_{I_0}|^2 dt
$$
from above by a multiple of $\Vert u \Vert_*^2$, but when $y \leq \delta/2^{N+1}$ we see
from  (\ref{smallJ}) that $\Vert u \Vert_*$
can be replaced with $\varepsilon$.
\end{proof}

For any $u \in B_p(\mathbb R)$,  by setting $\Lambda(u) = \mu_u$, we define a map $\Lambda$ on $B_p(\mathbb R)$. 
More explicitly,
\begin{equation}\label{dilatation}
\Lambda(u)(x,y)=\frac{\alpha_y\ast e^u(x)}{\beta_y\ast e^u(x)}
\end{equation}
for the functions $\alpha, \beta$ in (\ref{alphabeta}). Now we are ready for showing the main part of our result,
the holomorphy of $\Lambda$.

\begin{theorem}\label{complex}
There exists a neighborhood $U(B_p^{\mathbb R}(\mathbb R))$ of 
the subspace $B_p^{\mathbb R}(\mathbb R)$ of the real-valued functions in $B_p(\mathbb R)$ such that
$\Lambda: U(B_p^{\mathbb R}(\mathbb R)) \to \mathcal{L}_p(\mathbb{U})$ is holomorphic and the image of 
$U(B_p^{\mathbb R}(\mathbb R))$ under $\Lambda$ is contained in $\mathcal{M}_p(\mathbb{U}) \cap M_0(\mathbb U)$.
\end{theorem}

It should be pointed out that the statement of Theorem \ref{complex} and its proof are inspired by 
Shen and Tang \cite[Lemma 6.1]{ST}. This is originally in Takhtajan and Teo \cite[p.30]{TT}.
It is worthwhile to compare our arguments with theirs. They proved that 
$\Lambda: U_\delta(0) \subset B_2(\mathbb R) \to \mathcal{M}_2(\mathbb{U})$ is holomorphic 
for some small $\delta$-neighborhood $U_\delta(0)$ of the origin under the premise that $\Lambda(u)$ is the complex dilatation of the modified Beurling--Ahlfors extension due to Semmes \cite{Se}, while we show that the variant of the Beurling--Ahlfors extension by the heat kernel has the better property in the sense that $\Lambda$ can be defined in some neighborhood of
the entire $B^{\mathbb R}_p(\mathbb R)$. We also generalize their result to the case of the
$p$-integrable class $\mathcal M_p(\mathbb U)$ and the $p$-Besov space $B_p(\mathbb R)$. More importantly,
our strategy is that we first prove the holomorphy of $\Lambda$ in a certain larger domain
by the local boundedness of $\Lambda$ and the holomorphy of $\Lambda$ in a weak sense,
and then we see that $\Lambda$ is continuous from this stronger property. By the continuity of $\Lambda$,
we obtain a smaller domain whose image under $\Lambda$ is contained in the appropriate space.

\begin{proof}[Proof of Theorem \ref{complex}]
We first show that for each $u_0 \in B_p^{\mathbb R}(\mathbb R)$, $\Lambda$ is a G\^ateaux holomorphic
function from the $\delta$-neighborhood $U_\delta(u_0)$ of $u_0$ to $\mathcal{L}_p(\mathbb{U})$, where $\delta>0$ is
the constant chosen in Lemma \ref{bounded} depending on $u_0$. Namely, we prove that
for every $\tilde u \in U_\delta(u_0)$ and every non-trivial
$\tilde v \in B_p(\mathbb R)$, the function $\lambda(t) = \Lambda(\tilde u+t\tilde v)$ 
of $t \in \mathbb C$ is holomorphic in some neighborhood of $0 \in \mathbb{C}$ with the image in $\mathcal{L}_p(\mathbb{U})$. 
We choose $\epsilon>0$ with $2\epsilon < (\delta - \Vert \tilde u \Vert_{B_p})/\Vert \tilde v \Vert_{B_p}$
so that $\tilde u+t\tilde v \in U_\delta(u_0)$ when $|t| \leq 2\epsilon$. Then, by (\ref{dilatation}),
the complex-valued function $\lambda(t)(z)$ for each fixed $z \in \mathbb{U}$
is holomorphic on $|t| \leq 2\epsilon$. By using the Cauchy integral formula, we have
\begin{align*}
&\quad\ \lambda(t)(z) - \lambda(t_0)(z) 
- (t - t_0) \left.\frac{d}{dt}\right|_{t = t_0}\lambda(t)(z)\\
& = \frac{1}{2\pi i}\oint_{|\tau| = 2\epsilon} \lambda(\tau)(z) \left(\frac{1}{\tau - t} - \frac{1}{\tau - t_0} - \frac{t - t_0}{(\tau - t_0)^2}  \right)d\tau\\
&= \frac{(t - t_0)^2}{2\pi i}\oint_{|\tau| = 2\epsilon} \frac{\lambda(\tau)(z)}{(\tau - t_0)^2(\tau - t)} d\tau.
\end{align*}

Since $\tilde u+t \tilde v \in U_{\delta}(u_0)$, we see from Lemma \ref{bounded} that 
$\Vert \lambda(t)(z)  \Vert_{\infty} \leq M$ for all $|t| = 2\epsilon$, and then we have
\begin{align*}
&\quad\left\Vert \frac{ \lambda(t)(z) - \lambda(t_0)(z)}{t - t_0} - 
\left.\frac{d}{dt}\right|_{t = t_0}\lambda(t)(z) \right  \Vert_{\infty}\\
& \leq \frac{|t - t_0|}{2\pi\epsilon^3}\oint_{|\tau| = 2\epsilon} \Vert \lambda(\tau)(z) \Vert_{\infty} |d\tau|
\leq \frac{2M}{\epsilon^2} |t - t_0|. 
\end{align*}
Moreover, by Lemma \ref{bounded} again, we have
\begin{align*}
&\quad \iint_{\mathbb{U}} \frac{1}{y^2} \left|  \frac{ \lambda(t)(z) 
- \lambda(t_0)(z)}{t - t_0} - \left.\frac{d}{dt}\right|_{t = t_0}\lambda(t)(z) \right|^p dxdy\\
& \leq \frac{|t - t_0|^p}{(2\pi)^p \epsilon^3} \iint_{\mathbb{U}} \frac{1}{y^2} \left( \oint_{|\tau| = 2\epsilon} |\lambda(\tau)(z)| |d\tau| \right)^p dxdy\\
&\leq \frac{2^{p-1}|t - t_0|^p}{2\pi\epsilon^{4-p}} \oint_{|\tau| = 2\epsilon}  
\left( \iint_{\mathbb{U}} \frac{|\lambda(\tau)(z)|^p}{y^2} dxdy \right) |d\tau|
\leq \frac{(2M)^p}{\epsilon^{3-p}}|t - t_0|^p.
\end{align*}
Consequently, the limit
$$
\lim_{t\to t_0} \frac{ \lambda(t) -  \lambda(t_0)}{t - t_0} 
=  \left.\frac{d}{dt}\right|_{t = t_0}\lambda(t)
$$
exists in $\mathcal{L}_p(\mathbb{U})$ and thus $\Lambda$ is G\^ateaux holomorphic in $U_\delta(u_0)$. 

It is known that a locally bounded G\^ateaux holomorphic function is holomorphic in the sense that
it is Fr\'echet differentiable (see \cite[Theorem 14.9]{Ch}, \cite[Proposition 3.7]{Di} and \cite[Theorem 36.5]{Mu}).
Thus, we see that $\Lambda$ is holomorphic on some neighborhood $U(B_p^{\mathbb R}(\mathbb R))$ of 
$B_p^{\mathbb R}(\mathbb R)$. In particular, $\Lambda$ is continuous there.
By Theorem \ref{FKP} and Proposition \ref{bounded}, we have $\Lambda(B_p^{\mathbb R}(\mathbb R)) \subset \mathcal{M}_p(\mathbb U)$.
Since $\mathcal{M}_p(\mathbb U)$ is open in $\mathcal{L}_p(\mathbb U)$, the image of $U(B_p^{\mathbb R}(\mathbb R))$
under $\Lambda$ is contained in $\mathcal{M}_p(\mathbb U)$ 
by making the neighborhood $U(B_p^{\mathbb R}(\mathbb R))$ smaller if necessary.
Since we have $|\mu_{\tilde u}(x,y)| \leq C |\alpha_y\ast e^{\tilde u-\tilde u_{I(x,y)}}(x)|$ by (\ref{onlyalpha}),
Proposition \ref{M_0} implies that $\mu_{\tilde u} \in M_0(\mathbb U)$.
That is all what we have to prove.
\end{proof}

Finally, we verify that the mapping $F_u$ defined on $\mathbb R$ by \eqref{gamma} and
on $\mathbb U$ and $\mathbb L$ by formulae (\ref{F})  
and (\ref{lowerhalfplane}) respectively
for $u \in U(B_p^{\mathbb R}(\mathbb R))$ is a quasiconformal homeomorphism of $\mathbb C$.
Here, the neighborhood $U(B_p^{\mathbb R}(\mathbb R))$ should be taken smaller if necessary
so that the same result on $\mathbb L$ as Theorem \ref{complex} on $\mathbb U$ is also satisfied.
The spaces $M(\mathbb L)$, $M_0(\mathbb L)$, and $\mathcal M_p(\mathbb L)$ of Beltrami coefficients on $\mathbb L$ are defined in the same way.

\begin{theorem}\label{qcU}
For any $u \in U(B_p^{\mathbb R}(\mathbb R))$, the mapping $F_u$ defined on $\mathbb C$ is a
quasiconformal homeomorphism onto $\mathbb C$ such that its complex dilatations on $\mathbb U$ and on $\mathbb L$
are in ${\mathcal M}_p(\mathbb U)\cap M_0(\mathbb U)$ and in ${\mathcal M}_p(\mathbb L)\cap M_0(\mathbb L)$,
respectively, and 
they both depend holomorphically on $u$. 
\end{theorem}

\begin{proof}
We can choose a sequence $\{u_j\}$ in $U(B_p^{\mathbb R}(\mathbb R))$ such that 
each $u_j$ is continuous and compactly supported and $u_j$ converges to $u$ in $B_p(\mathbb R)$.
Indeed, for each $j \in \mathbb N$, we set $u_j=\eta_{1/j} \ast (u1_{[-j,j]})$, where $\eta_{1/j}$ is
the mollifier defined in (\ref{mollifier}). We see that $u1_{[-j,j]}$ converges to $u$ in $B_p(\mathbb R)$
as $j \to \infty$. Moreover, it is known (see \cite[Proposition 14.5]{Leo}) that for $\tilde u \in B_p(\mathbb R)$ in general, 
$\eta_{1/j} \ast \tilde u$
converges to $\tilde u$ in $B_p(\mathbb R)$. This shows that $u_j \to u$ in $B_p(\mathbb R)$.
See also \cite[Theorem 5.3]{BGV} for the claim that smooth and compactly supported functions are dense in 
the $p$-Besov space.

By the same argument as in the proof of Theorem \ref{smallcase}, we see that $F_{u_j}$ is a
quasiconformal homeomorphism of $\mathbb C$. Moreover, by Theorem \ref{complex}
also applied on $\mathbb L$, the complex dilatation $\mu_{u_j}$ of $F_{u_j}$
converges to $\mu_u$ of $F_u$ in the $L^\infty$ norm.
Let $\widetilde F$ be the quasiconformal homeomorphism of $\mathbb C$ whose complex dilatation is $\mu_u$.
If we normalize $F_{u_j}$, $F_u$ and $\widetilde F$ suitably, then $F_{u_j}$ converges locally uniformly to $\widetilde F$.
Since $u_j$ converges to $u$ in $B_p(\mathbb R)$, we see that $e^{u_j}$ converges to $e^u$ locally in $L^1$.
Then, by the definition of $F_{u_j}$ and $F_u$ in (\ref{F}) and (\ref{lowerhalfplane}), $F_{u_j}$ converges to $F_u$.
Therefore, $F_u$ coincides with $\widetilde F$, which proves that $F_u$ is a quasiconformal homeomorphism of $\mathbb C$.

The rest of the statements has been shown in Theorem \ref{complex} if we extend it also to $\mathbb L$.
\end{proof}

\begin{remark}
There are remaining problems of showing that for every $\tilde u \in U(B_p^{\mathbb R}(\mathbb R))$ in Theorems \ref{complex}
and \ref{qcU},
the quasiconformal homeomorphism $F_{\tilde u}$ of $\mathbb U$ (and $\mathbb L$) onto its image is bi-Lipschitz with respect to the hyperbolic 
metrics as well as $\frac{|\mu_{\tilde u}(x,y)|^2}{y}dxdy$ is a vanishing Carleson measure.
On the contrary, we can obtain a little stronger consequence about the quasiconformality; $F_{\tilde u}$ is asymptotically conformal on
$\mathbb U$ (and $\mathbb L$). This means that its complex dilatation $\mu_{\tilde u}$ satisfies
$$
\inf\,\{\Vert \mu_{\tilde u}|_{\mathbb U \backslash K} \Vert_{\infty}: K\subset \mathbb U \;{\rm compact}\} = 0. 
$$
This can be proved also on the unit disk $\mathbb D$
if we consider the same problem there. 
\end{remark}

By Theorem \ref{qcU},
we can extend $\Lambda$ in Theorem \ref{complex} to a holomorphic map
$$
\widetilde \Lambda:U(B^{\mathbb R}_p(\mathbb R)) \to \mathcal M_p(\mathbb U) \times \mathcal M_p(\mathbb L)
$$
such that the Beltrami coefficients in the image correspond to the quasiconformal homeo\-morphism $F_u$ extending $\gamma_u$.
Then, this satisfies the properties of Theorem \ref{thm5} and 
Corollaries \ref{specialcase} and \ref{cor6} in Section 1.

\section{Appendix: The $p$-Weil--Petersson class on the unit circle}

Let $W_p(\mathbb S)$ denote the set of all quasisymmetric homeomorphisms $g$
of the unit circle $\mathbb S$ onto itself that has a quasiconformal extension $G$ to the unit disk $\mathbb D$ whose complex dilatation $\nu$ is $p$-integrable in the hyperbolic metric, namely, 
$$
\iint_{\mathbb D}\frac{|\nu(w)|^p}{(1-|w|^2)^2} du dv < \infty. 
$$
We call $W_p(\mathbb S)$ the {\it $p$-Weil--Petersson class} on $\mathbb S$.
The class $W_2(\mathbb S)$ for $p=2$ was first introduced and studied by Cui \cite{Cu} 
and then investigated by Takhtajan and Teo \cite{TT}. 
For $p \geq 2$, $W_p(\mathbb S)$ appeared in Guo \cite{Gu} 
(see aslo \cite{Ya}). 

Shen \cite{Sh18} characterized intrinsically the elements in the Weil--Petersson class
$W_2(\mathbb S)$ without using quasiconformal extensions, which solved the problem proposed in \cite{TT}. 
Later on, Tang and Shen \cite{TS} generalized this result to any $p \geq 2$.  
Let $B_p(\mathbb S)$ 
be the $p$-Besov space of all locally integrable functions $v$ on $\mathbb{S}$ with $\Vert v \Vert_{B_p}<\infty$,
where 
$$
\Vert v \Vert_{B_p}^{p} = \int_{\mathbb S} \! \int_{\mathbb S} \frac{|v(z)-v(w)|^p}{|z-w|^2}\frac{|dz|}{2\pi}\frac{|dw|}{2\pi}=
\int_0^1 \!\int_0^1\frac{|v(e^{2\pi ix}) - v(e^{2\pi iy})|^p}{|e^{2\pi ix} - e^{2\pi iy}|^2} dx dy.
$$
Then, the results of \cite{Sh18} and \cite{TS} can be stated as follows.

\begin{theorem}\label{circlecase}
Let $g$ be a sense-preserving homeomorphism of $\mathbb S$ onto itself. 
Then, $g$ is absolutely continuous and $\log g'$ belongs to the $p$-Besov space $B_p(\mathbb S)$ $(p \geq 2)$ if and only if $g$ 
belongs to the $p$-Weil--Petersson class $W_p(\mathbb S)$.
\end{theorem}

Recently, Wu, Hu and Shen \cite{WHS} gave an alternative proof for the case of $p=2$ by exporting 
the result on $\mathbb R$ obtained by
the modified Beurling--Ahlfors extension due to Semmes \cite{Se}, 
which is much different from the method given previously in \cite{Sh18}. 

The purpose of this appendix is to show that the variant of the Beurling--Ahlfors extension by the heat kernel,  
which is translated to the setting of the unit disk, also yields the desired quasiconformal extension.  
This in particular gives an alternative proof of the only-if part of Theorem \ref{circlecase} for a general $p>1$.
As the extension used in \cite{WHS} is valid only for such $g$ with small norm,
the decomposition of $g$ and the composition of such extensions are required,
but our method gives a straight extension and certain properties of the complex dilatation are thus inherited.

For a sense-preserving homeomorphism $g:\mathbb S \to \mathbb S$ with $\log g' \in B_p(\mathbb S)$,
our quasiconformal extension $G:\mathbb D \to \mathbb D$ is precisely defined as follows.
We first note that 
$\log g' \in B_p(\mathbb S)$ implies that $v = \log |g'|$ belongs to the subspace $B^{\mathbb R}_p(\mathbb S)$
of real-valued functions. In fact, 
$\Vert v \Vert_{B_p} \leq \Vert \log g' \Vert_{B_p}$.
In addition, by the same proof of Proposition \ref{BMOnorm} applied to the unit circle case,
we have $v \in {\rm VMO}(\mathbb S)$ and $\Vert v \Vert_* \leq \Vert v \Vert_{B_p}$.

We take a continuous lift $f: \mathbb{R} \to \mathbb{R}$ of $g$ satisfying 
that $g(e^{2\pi ix}) = e^{2\pi i f(x)}$ for $x \in \mathbb{R}$. 
Let $u(x) = \log f'(x)=\log |g'(e^{2\pi ix})|$. 
This satisfies $u(x+1) = u(x)$ for any $x \in \mathbb{R}$. From this periodicity, 
we also see that $u \in {\rm VMO}(\mathbb R)$, but
$u$ does not necessarily belong to $B_p(\mathbb R)$.
Moreover, we can verify that$\Vert v \Vert_* \leq \Vert u \Vert_* \leq 3\Vert v \Vert_*$ (see \cite[Lemma 2.2]{Pa}).

Let $F: \mathbb{U} \to \mathbb{U}$ be the variant of the Beurling--Ahlfors extension by the  
heat kernel such that $F|_{\mathbb{R}} = f$, 
and let $\mu(z) = F_{\bar z}/F_z$ be the complex dilatation of $F$. Since $f(x+1) = f(x) +1$, 
we see from definition (\ref{F}) that the quasiconformal extension $F$ of $f$ satisfies $F(z+1) = F(z) + 1$. Thus, $F$ can be projected to a quasiconformal homeomorphism $G$ of the punctured disk $\mathbb D\backslash \{0\}$ onto itself such that $G(e^{2\pi iz}) = e^{2\pi iF(z)}$ for $z \in \mathbb U$ and $G|_{\mathbb S} = g$. 
Clearly, $G$ can be extended quasiconformally
to $0$, and the resulting mapping from $\mathbb D$ onto itself is still denoted by $G$. 

Concerning this quasiconformal extension $G$ of $g$, we prove the following properties.
The properties 
that the complex dilatation $\nu$ on $\mathbb D$ vanishes at the boundary and induces a vanishing Carleson measure are defined similarly to the case of $\mathbb U$. 

\begin{theorem}\label{Delta}
Let $g$ be a sense-preserving absolutely continuous homeomorphism of the unit circle $\mathbb S$ onto itself with  
$v=\log |g'|\in B_p^{\mathbb R}(\mathbb S)$.  
Then, the complex dilatation $\nu$ of the quasiconformal homeomorphism $G$ of $\mathbb D$ onto itself
defined above satisfies
$$
\iint_{\mathbb D}\frac{|\nu(w)|^p}{(1 - |w|^2)^2} dudv \leq C_p(v) \Vert v \Vert_{B_p}^p
$$
for a locally bounded constant $C_p(v)>0$ depending on $v \in B_p^{\mathbb R}(\mathbb S)$.
In particular, $g \in W_p(\mathbb S)$. Moreover, $\nu$ vanishes at the boundary,
and 
$$
\frac{1}{1-|w|^2}|\nu(w)|^2dudv
$$
is a vanishing Carleson measure on $\mathbb D$.
\end{theorem}

\begin{proof} 
The complex dilatation $\nu(w) = G_{\bar w}/ G_w$ $(w \in \mathbb D)$ satisfies 
$\nu(e^{2\pi i z})\overline{e^{2\pi iz}}/e^{2\pi iz} = -\mu(z)$ for 
the complex dilatation $\mu(z)=F_{\bar z}/F_z$
$(z \in \mathbb{U})$, and 
in particular, $\Vert \nu \Vert_{\infty} = \Vert \mu \Vert_{\infty}$. 
We fix some constant $r_0$ with $e^{-\pi} < r_0 < 1$ so that $c = \frac{1}{2\pi}\log\frac{1}{r_0} < \frac{1}{2}$. Noting that $e^{2\pi y} - 1 \geq 2\pi y$ for $y > 0$, 
we have 
\begin{equation}\label{prep}
\begin{split}
\iint_{r_0 < |w| < 1}\frac{|\nu(w)|^p}{(1 - |w|^2)^2} dudv
& = 4\pi^2\int_0^1 dx \int_0^c \frac{|\nu(e^{2\pi iz})|^p}{(1 - |e^{2\pi iz}|^2)^2} |e^{2\pi iz}|^2 dy\\ 
& \leq \int_0^1 dx \int_0^c \frac{|\mu(z)|^p}{y^2} dy.\\
\end{split}
\end{equation}
Here, we can estimate $|\mu(z)|^p$ by (\ref{p}).
Then, similarly to (\ref{compute}), we obtain by using the periodicity $u(x+1) = u(x)$ and $c<1/2$ that
\begin{equation}\label{p+}
\begin{split}
\int_0^1 dx\int_0^c \frac{L_n}{y^2}dy&=\frac{1}{2^{n+1}}\int_0^1 dx\int_0^c \frac{1}{y^3}dy\int_{-2^{n}y}^{2^{n}y} |u(x+t) - u(x)|^p dt\\
&=\frac{1}{2^{n+1}} \int_0^1 dx\int_{-2^nc}^{2^nc} |u(x+t)- u(x)|^p dt \int_{2^{-n}t}^c \frac{1}{y^3}dy\\
&\leq 2^{n-2} \int_{0}^{1} dx \int_{-2^nc}^{2^nc}\frac{|u(x+t) - u(x)|^p}{t^2} dt\\
&\leq 2^{2n-2} \int_{0}^{1} dx \int_{-1/2}^{1/2}\frac{|u(x+t) - u(x)|^p}{t^2} dt
= 2^{2n-2} \Vert v \Vert_{B_p}.\\
\end{split}
\end{equation}
Hence, by (\ref{p}), (\ref{prep}), and (\ref{p+}), we conclude that
\begin{equation*}
\begin{split}
\iint_{r_0 < |w| < 1}\frac{|\nu(w)|^p}{(1 - |w|^2)^2} dudv
&\leq C_p(v)\sum_{n=0}^\infty \frac{1}{e^{2^n}}\int_0^1 dx \int_0^c \frac{L_n+L_0}{y^2} dy\\
&\leq C_p(v)\sum_{n=0}^\infty \frac{2^{2n-2}+2^{-2}}{e^{2^n}}\Vert v \Vert_{B_p}
\end{split}
\end{equation*}
for a locally bounded constant $C_p(v)>0$ depending on $v \in B_p^{\mathbb R}(\mathbb S)$.

On the other hand, we see from Proposition \ref{mu<BMO} that
$\Vert \mu \Vert_\infty \lesssim \Vert u \Vert_*$ in any case without restricting $\Vert u \Vert_*$
to be small since $\Vert \mu \Vert_\infty<1$ holds in the present situation. 
Therefore, 
$$
\Vert \nu \Vert_{\infty}=\Vert \mu \Vert_{\infty} \lesssim \Vert u \Vert_*
\lesssim  \Vert v \Vert_* \leq \Vert v \Vert_{B_p},
$$ 
and hence,
$$
\iint_{|w| < r_0} \frac{|\nu(w)|^p}{(1 - |w|^2)^2} dudv \leq \frac{\pi}{(1 - r_0^2)^2}\Vert \nu \Vert_{\infty}^p \lesssim \Vert v \Vert_{B_p}^p. 
$$
Consequently, 
\begin{align*}
\iint_{\mathbb D}\frac{|\nu(w)|^p}{(1 - |w|^2)^2} dudv &= \iint_{r_0 < |w| < 1}\frac{|\nu(w)|^p}{(1 - |w|^2)^2} dudv +
 \iint_{ |w| < r_0}\frac{|\nu(w)|^p}{(1 - |w|^2)^2} dudv\\
& \leq C_p(v) \Vert v \Vert_{B_p}^p.
\end{align*}

To see that $\nu$ vanishes at the boundary, it suffices to see that
$\mu$ vanishes at the boundary. This was already proved in Proposition \ref{M_0}.
To see that $\frac{1}{1-|w|^2}|\nu(w)|^2dudv$ 
is a vanishing Carleson measure on $\mathbb D$, it suffices to see that
$\frac{1}{y}|\mu(z)|^2dxdy$ is a vanishing Carleson measure on $\mathbb U$.
This was already proved in Theorem \ref{FKP}.
This completes the proof of Theorem \ref{Delta}. 
\end{proof}

\end{document}